\documentclass[12pt,reqno,papersize]{amsart}
\usepackage{amsmath,amssymb}
\usepackage{graphicx,color,xcolor}
\usepackage{latexsym,amssymb}
\usepackage{mathrsfs}

\newtheorem{thm}{Theorem}[section]

\newtheorem{lem}[thm]{Lemma}
\newtheorem{rem}[thm]{Remark}

\renewcommand{\div}{\text{div}}

\numberwithin{equation}{section}
\numberwithin{thm}{section}

\begin{document}
\begin{center}\large \bf 
On self-similar solutions to degenerate compressible Navier-Stokes equations
\end{center}

\vskip5mm

\footnote[0]
{
{\it Mathematics Subject Classification} (2010): 35C06, 76N10

{\it 
Keywords}: Compressible Navier-Stokes equations; self-similar solutions.

{\it Acknowledgements}: 
P. Germain was partially supported by the NSF grant DMS-1501019. 
T. Iwabuchi was supported by JSPS KAKENHI Grant 17H04824. 

{\it Addresses}: 
Pierre Germain, Courant Institute of Mathematical Sciences, New York University, 251 Mercer Street, New York, N.Y. 10012-1185, USA, {\tt pgermain@cims.nyu.edu}

Tsukasa Iwabuchi, Mathematical Institute, Tohoku University, 
Sendai City, 980-8578, JAPAN, 
{\tt t-iwabuchi@tohoku.ac.jp}

Tristan L\'eger, Courant Institute of Mathematical Sciences, New York University, 251 Mercer Street, New York, N.Y. 10012-1185, USA, {\tt tleger@cims.nyu.edu}
}


\begin{center}

Pierre GERMAIN, \quad Tsukasa IWABUCHI, \quad Tristan L\'EGER 

\vskip2mm


\end{center}

\vskip5mm

\begin{center}
\begin{minipage}{135mm}
\footnotesize
{\sc Abstract. } 
We study cavitating self-similar solutions to compressible Navier-Stokes equations with degenerate density-dependent viscosity. We prove both existence of expanders and non-existence of small shrinkers. 

\end{minipage}
\end{center}

\tableofcontents

\section{Introduction}
\subsection{The equations}
In this article we consider the compressible Navier-Stokes-Fourier system, which describes the motion of a heat conducting gas:
\begin{equation*}
\begin{cases}
\partial_t \rho + {\rm div \, } \big( \rho u \big) 
= 0 , 
\\
\displaystyle 
\partial _t (\rho u) + 
{\rm div \,} \big( \rho u \otimes u \big) + \nabla \pi 
 = \div \big( \tau \big),
\\
\displaystyle 
\partial_t 
\Big[ \rho \Big( \frac{|u|^2}{2} + e \Big) \Big] 
+ {\rm div} 
\Big[ u \Big( \rho \Big(\frac{|u|^2}{2} +e \Big) 
      + \pi \Big) \Big]
- {\rm div \,} q
\\
\qquad \qquad \qquad \qquad \qquad \qquad \qquad = {\rm div \,} (\tau \cdot u ) ,
\end{cases}
\end{equation*}
where $\rho:\mathbb{R} \times \mathbb{R}^d \rightarrow \mathbb{R}$ denotes the density of the fluid, $u:\mathbb{R} \times \mathbb{R}^d \rightarrow \mathbb{R}^d$ its velocity field, $\pi:\mathbb{R} \times \mathbb{R}^d \rightarrow \mathbb{R}$ its pressure, $\tau:\mathbb{R} \times \mathbb{R}^d \rightarrow \mathbb{R}^{d^2}$ its stress tensor, $e:\mathbb{R} \times \mathbb{R}^d \rightarrow \mathbb{R}$ its internal energy and $q:\mathbb{R} \times \mathbb{R}^d \rightarrow \mathbb{R}^d$ its internal energy flux.
\\
\indent We will assume that $e$ is given by Joule's law, that is $e=C_V \theta$ ($C_V$ is the heat constant, and $\theta(t,x)$ the temperature of the gas). We also take $q$ to be proportional to the gradient of $\theta,$ in accordance with Fourier's law: $q = - \kappa \nabla \theta,$ where $\kappa>0$ denotes the thermal conductivity. The pressure will be given by the ideal gas law: $\pi = \rho R \theta.$ We will also restrict our attention to newtonian gases, for which $\tau$ is given by
\begin{align*}
\tau = \lambda \textrm{div} u~ \textrm{Id} + 2\mu D(u), \ \ D(u):= \big( \frac{\partial_i u_j + \partial_j u_i}{2} \big)_{1 \leqslant i \leqslant d, 1 \leqslant j \leqslant d}, 
\end{align*}
where the Lam\'{e} coefficients $\lambda, \mu$ are such that $$
\mu>0, \ \ 2\mu+d\lambda \geqslant 0.
$$
Moreover it is physically relevant (from the point of view of kinetic theory for example) to consider the coefficients $\lambda$ and $\mu$ that depend on the density $\rho.$ More precisely, we postulate the following laws:
\begin{align}\label{visc-law}
\lambda(\rho) = \lambda_0 \rho^{\alpha},  \ \  \ \ \mu(\rho)= \mu_0 \rho^{\alpha}, 
\end{align} 
where $0 < \alpha \leqslant 1.$ \\
The equations read
\begin{equation} \label{eq:cNS} \tag{cNS}
\begin{cases}
\partial_t \rho + {\rm div \, } \big( \rho u \big) 
= 0 , 
\\
\displaystyle 
\partial _t (\rho u) + 
{\rm div \,} \big( \rho u \otimes u \big) + \nabla ( \rho R \theta ) 
 = \div \big( \lambda_0 \rho^{\alpha} \div u~ \textrm{Id} + \mu_0 \rho^{\alpha} (\nabla u + \nabla u^T )     \big),
\\
\displaystyle 
\partial_t 
\Big[ \rho \Big( \frac{|u|^2}{2} + C_V \theta \Big) \Big] 
+ {\rm div} 
\Big[ u \Big( \rho \Big(\frac{|u|^2}{2} + C_V \theta \Big) 
      + \rho R \theta \Big) \Big]
- \kappa \Delta \theta
\\
\qquad \qquad \qquad \qquad \qquad \qquad \qquad = {\rm div \,} (\lambda_0 \rho^{\alpha} ({\rm div \,} u) u + \mu_0 \rho^{\alpha}  (\nabla u + \nabla u^T ) \cdot u ) ,
\end{cases}
\end{equation}
Note that the system is scaling invariant: if $(\rho,u,\theta)$ denotes a solution, then $(\rho_{\lambda},u_{\lambda},\theta_{\lambda}),$ where
\begin{align*}
\rho_{\lambda}(t,x) := \rho(\lambda^2 t, \lambda x), \ \ \ u_{\lambda}(t,x) := \lambda u (\lambda^2 t , \lambda x), \ \ \ \theta_{\lambda}(t,x) := \lambda^2 \theta(\lambda ^2 t , \lambda x)
\end{align*}
is also a solution.
\\
\indent This leads us to considering self-similar solutions to \eqref{eq:cNS}. In this paper, we will therefore restrict our attention to self-similar, radially symmetric solutions. Moreover since a well-known difficulty in the case of degenerate viscosity is the presence of vacuum, we will also assume that our solutions exhibit cavitation at the origin. More precisely, we study both expanders of the form
\begin{align}\label{expanders}
\rho(t,x) = P \bigg( \frac{\vert x \vert}{\sqrt{t}} \bigg), & \ \ \  
 u(t,x) = \frac{1}{\sqrt{t}} U \bigg( \frac{\vert x \vert}{\sqrt{t}} \bigg) \frac{x}{\vert x \vert}, \\ \notag \theta(t,x) &= \frac{1}{t} \Theta \bigg(\frac{\vert x \vert}{\sqrt{t}} \bigg),
\end{align}
and shrinkers:
\begin{align}\label{shrinkers}
\rho(t,x) = P \bigg( \frac{\vert x \vert}{\sqrt{T-t}} \bigg), & \ \ \  
 u(t,x) = \frac{1}{\sqrt{T-t}} U \bigg( \frac{\vert x \vert}{\sqrt{T-t}} \bigg) \frac{x}{\vert x \vert}, \\ \
\notag \theta(t,x) &= \frac{1}{T-t} \Theta \bigg(\frac{\vert x \vert}{\sqrt{T-t}} \bigg),
\end{align}
where $T>0.$ We will require that $P(0)=0$ to signify that there is cavitation at the origin. 

\subsection{Background}
\subsubsection{Weak solutions}
The system \eqref{eq:cNS} and related models have been extensively studied in the literature. In the constant coefficient case, A. Matsumura and T. Nishida proved global existence for small data in \cite{MN-1980}. Note that in this work the density is bounded away from 0, therefore no vacuum is present. P.-L. Lions constructed weak solutions for barotropic compressible Navier-Stokes with large initial data and possibly vacuum states in \cite{Lions-1998}. His result was subsequently refined and extended to the full system by E. Feireisl, A. Novotn\'y and H. Petzeltov\'a in \cite{FNP-2001} and E. Feireisl in \cite{F-2004}. \\
\indent In the case of degenerate density-dependent viscosity (for example of the type \eqref{visc-law}), the situation becomes more involved. A work of major importance in this direction is the paper by D. Bresch, B. Desjardins and C.-K. Lin (\cite{BDL}), where this problem is treated for various compressible models, including the shallow water model away from vacuum. Recently A. Vasseur and C. Yu were able to construct weak solutions to the barotropic compressible Navier-Stokes with vacuum in \cite{VY}.

\subsubsection{Self-similar solutions}
Self-similar solutions have been the subject of investigation since the seminal work of J. Leray \cite{Le} on the incompressible Navier-Stokes equation. Indeed he noticed that the existence of a backward solution of this type would imply singularity formation. Forward self-similar solutions are also if interest since they are expected to describe the continuation of the backward solution after the singular time. Such small expander type solutions were constructed by Cannone and Planchon in \cite{CP}. Recently V. \v Sver\'{a}k and H. Jia were able to construct solutions to the incompressible Navier-Stokes equation with large self-similar initial data in \cite{JS}. Regarding backward self-similar solutions, J. Ne\v cas, M. R\r u\v zi\v cka, V. \v Sver\'{a}k proved their non existence in the natural energy class of the Navier-Stokes equation. This result was improved by Tsai who showed that in fact this result still holds if only local energy inequalities are assumed.
\\
\indent 
For compressible Navier Stokes equations, Z. Guo and S. Jiang showed in \cite{GJ-2006} that in the 1D isothermal barotropic case, there exist neither forward nor backward self-similar solutions. In the case of \eqref{eq:cNS} with constant viscosity, the first two authors constructed expanders both with and without cavitation in \cite{GI}. Finally in the companion paper \cite{GIL}, we prove non existence of small shrinkers, still in the constant coefficient case. 
 
\subsection{Results}
We prove two types of results in this paper: existence of small expanders in Section \ref{fwrd}, and non-existence of small shrinkers in Section \ref{backwrd}. The construction of forward solutions follows the same pattern as in the paper of the first two authors \cite{GI}, although the proof is more involved due to the degeneracy of the coefficients at vacuum. It requires a more precise control of the velocity profile near the origin. For technical reasons, we distinguish between the cases $0<\alpha<1$ and $\alpha = 1.$ 
In the first situation ($0<\alpha<1$), here is simplified statement of our result:
\begin{thm}
Let $d \geqslant 3.$ Assume $2\mu_0+d\lambda_0=0.$ Then if $$A + 
P_\delta +
\delta +  P_\delta ^{1-\alpha} \delta^2 + \frac{P_\delta^{1-\alpha} \Theta_0}{A} 
+ \frac{A\delta}{\Theta_0} + \frac{A^2}{\Theta_0} + P_{\delta} A $$ is small enough, there exists a solution to \eqref{eq:cNS} of the form \eqref{expanders} such that
\begin{align*}
P(\delta) = P_{\delta}, U(0)=0, U'(0)=A, \Theta(0)=\Theta_0>0, \Theta'(0)=0. 
\end{align*}
\end{thm}
In the second case ($\alpha=1$), we prove
\begin{thm}
Let $d \geqslant 3.$ Assume $2\mu_0+d\lambda_0>0.$ Then if $ A \log \delta ^{-1} +  P_{\delta} + \delta $ is small enough, there exists a solution to \eqref{eq:cNS} of the form \eqref{expanders} such that
\begin{align*}
P(\delta) = P_{\delta}, U(0)=0, U'(0)=A, \Theta(0)=\frac{2\mu_0+\lambda_0}{R}A>0, \Theta'(0)=0. 
\end{align*}
\end{thm}
Note that for $0<\alpha<1,$ we construct a three parameter family of solutions, and for $\alpha=1$ we obtain a two dimensional family (the choice of $\Theta_0$ is not free). In both cases we also obtain a precise description of the shape of the profiles. Details can be found in the full statements of these theorems in Section \ref{fwrd}.
\\
\indent In the second part of the paper, we show that \eqref{eq:cNS} does not have small solutions of shrinker type \eqref{shrinkers}. We prove the following result (in simplified form here):
\begin{thm}
Consider a cavitating solution to \eqref{eq:cNS} of the form \eqref{shrinkers}. Let $\varepsilon>0.$ \\
Assume that 
\begin{align*}
\sup_{r>0} \bigg( \langle r \rangle^2 \Theta + P^{1-\alpha}+ \bigg \vert \frac{U}{r \Theta} \bigg \vert \bigg)  + \sup_{r>\varepsilon} \bigg \vert \frac{U'}{r \Theta'} \bigg \vert  
\end{align*}
is sufficiently small. \\
Then $U \equiv \Theta \equiv 0, P = Constant.$
\end{thm}
We show a similar result in the companion paper \cite{GIL}, when the Lam\'{e} coefficients are constant. Note however that the two approaches are different. Indeed some key cancellations are no longer available in the density-dependent viscosity case. Therefore we develop a different method tailored to our setting of solutions exhibiting cavitation in the present work. 

\section{Existence of expanders} \label{fwrd}
In this section we construct expander solutions to \eqref{eq:cNS}, that is solutions of the type \eqref{expanders}. Plugging this ansatz into \eqref{eq:cNS}, we obtain the following system of ODEs for the profiles $P,U$ and $\Theta:$
\begin{equation}\label{ODEs:fwrd}
\begin{cases}
\displaystyle 
-\frac{1}{2} r P ' + P' U 
+ P \Big( U' + \frac{d-1}{r} U \Big) 
= 0 , 
\\[5mm]
\begin{split}
\displaystyle 
-\frac{1}{2} 
 P  U 
& 
- \frac{1}{2} r (P U)' 
+ (P U ^2) ' + \frac{d-1}{r} P U^2 
+(P R \Theta)'
\\
= 
& 
(2\mu_0 + \lambda_0) P^{\alpha}
 \Big( U'' + \frac{d-1}{r} U' - \frac{d-1}{r^2} U \Big) 
 \\&+ (2 \mu_0 + \lambda_0) \big(P^{\alpha})' U' + \lambda_0 \big(P^{\alpha} \big)' \frac{d-1}{r} U ,
\end{split}
\\[10mm]
\displaystyle 
\begin{split}
-P 
 \Big( \frac{U^2}{2}
&  + C_V \Theta \Big) 
- \frac{1}{2} r 
  \Big( P \Big( \frac{U^2}{2} + C_V \Theta \Big) \Big) '
+ \Big( U P \Big( \frac{U^2}{2} + C_V \Theta \Big) 
        + U P R \Theta \Big) '
\\
&
+ \frac{d-1}{r} 
  \Big( U P \Big( \frac{U^2}{2} + C_V \Theta \Big)
       + U P R \Theta
  \Big)
- \kappa \Big( \Theta '' + \frac{d-1}{r} \Theta ' \Big)
\\
= 
& 
2\mu_0 P^{\alpha} \Big( (U')^2 + \frac{d-1}{r^2} U^2 \Big) 
+ \lambda_0 P^{\alpha} \Big( U' + \frac{d-1}{r} U \Big) ^2
\\
& 
+ (2\mu_0 + \lambda_0) P^{\alpha}
 \Big( U'' + \frac{d-1}{r}U' - \frac{d-1}{r^2} U \Big) U
\\
&
+ (2 \mu_0 + \lambda_0) \big(P^{\alpha}\big)' U'U + \lambda_0 \big(P^{\alpha} \big)' \frac{d-1}{r} U^2.  
\end{split}
\end{cases}
\end{equation}
As announced in the introduction, we distinguish two cases for the existence of expanders: when $0<\alpha<1$ and when $\alpha=1.$ There are substantial differences between the two: we do not use the same integro-differential formulation for both, and require different conditions on the parameters. Most strikingly, $\Theta_0$ is free for $0<\alpha<1$ but not for $\alpha=1.$

\subsection{Existence when $0<\alpha<1$} \label{frwd-0}
In this section, we show the following:
\begin{thm}\label{main-fwrd-0}
Let $d \geqslant 3$ and $0<\alpha <1.$ Fix $(C_V, \kappa, R, \mu_0, \lambda_0) \in (0,\infty)^{5}$ such that $2\mu_0+d\lambda_0 = 0.$  \\
Then there exists a constant $C(C_V, \kappa, R, \mu_0, \lambda_0) := C$ such that if 
\begin{align*}
A + 
P_\delta +
\delta +  P_\delta ^{1-\alpha} \delta^2 + \frac{P_\delta^{1-\alpha} \Theta_0}{A} 
+ \frac{A\delta}{\Theta_0} + \frac{A^2}{\Theta_0} + P_{\delta} A < C,
\end{align*}
there exists a solution $(P,U,\Theta) \in \mathcal{C}^{\frac{dA}{1/2-A}} \times \mathcal{C}^1 ([0,\infty)) \times  \mathcal{C}^1([0,\infty))$ to \eqref{ODEs:fwrd} such that 
\begin{align*}
P(0)=0, P(\delta) = P_{\delta}, U(0)=0, U'(0)=A, \Theta(0) = \Theta_0 >0, \Theta'(0)=0.
\end{align*}
Near $0,$ the behavior of the solution is given by (here $0<\varepsilon<1-\alpha$):
\begin{align*}
P(r) &= P_{\delta} \big( \frac{r}{\delta} \big)^{\frac{dA}{1/2-A}} + O \bigg( \big( \frac{r}{\delta} \big)^{1+\frac{dA}{1/2-A} +(1-\alpha-\varepsilon)dA } \bigg) , \\
U(r) &=   Ar + O(r^{1+(1-\alpha-\varepsilon)dA}), \\
\Theta(r) &= \Theta_0 + O(r^2).
\end{align*}
Moreover the profiles satisfy the following global bounds:
\begin{align*}
P(r) & \simeq P_{\delta} \min \Bigg[1, \bigg( \frac{r}{\delta} \bigg)^{\frac{2dA}{1-2A}} \Bigg],\\
\vert U(r) \vert \lesssim \frac{Ar}{(1+P_{\delta}^{1/2-\alpha/2}r)^2}, & \ \ \ \ \  \vert U'(r) \vert \lesssim \frac{A}{(1+P_{\delta}^{1/2-\alpha/2}r)^2}, \\
0 \leqslant \Theta(r) \lesssim \frac{1}{(1+\sqrt{P_{\delta}}r)^2},& \ \ \ \ \ \vert \Theta'(r) \vert \lesssim \frac{\sqrt{P_{\delta}}r}{(1+\sqrt{P_{\delta}}r)^2}.
\end{align*}
Finally as $r \to +\infty$, there exist $P_{\infty} >0,  U_{\infty}>0, \Theta_{\infty} >0$ such that
\begin{align*}
P(r) = P_{\infty} + O \big( \frac{1}{r^2}\big), \ \ \ U(r) = \frac{U_{\infty}}{r} +  O \big( \frac{1}{r^3}\big), \ \ \  \Theta(r) = \frac{\Theta_{\infty}}{r^2} + O \big( \frac{1}{r^4}\big).
\end{align*}
\end{thm}
The proof is modelled after that of a similar result in \cite{GI} (Section 4): first we find an integro-differential formulation of the problem. Then we construct the solution locally near 0 using a fixed-point argument. After that, we prove global existence, and finally we study the asymptotic behavior of the solutions. 
\\
\indent Note that the main difference compared to \cite{GI} is in the derivation of the integro-differential equation and the local existence part of the proof. These steps are more involved in the present paper due to the degeneracy of the Lam\'{e} coefficients at 0, and the presence of vacuum.

\subsubsection{Integro-differential formulation}
We are aiming at constructing solutions such that for $0 < A < 1/2$ and $\Theta _0 > 0,$ 
\[
P(r)  = O (r^{\frac{2d A}{1- 2A}}), \quad 
U(r) = A r + O(r^{1+(1-\alpha-\varepsilon) \frac{2d A}{1-2 A}}), \quad 
\Theta (r) = \Theta _0 + O(r^2), 
\quad r \to 0 .
\]
First starting with the equation on $P,$ we directly obtain after integration that
\begin{align*}
P(r) = P_{\delta} e^{V(r)-V(\delta)},
\end{align*}
where 
\begin{align}\label{def-V}
V(r) = \int_0 ^r \frac{U' + \frac{d-1}{r_1}U}{\frac{1}{2}r_1 - U} dr_1.
\end{align}
Next we move to the equation satisfied by $U:$ 
\[
\begin{split}
\displaystyle 
-\frac{1}{2} PU - \frac{1}{2} r (PU)'
& 
+ (P U^2)' + \frac{d-1}{r} PU^2 
+(P R \Theta)'
\\
= 
& 
(2\mu_0 + \lambda_0) 
\Big\{ P^{\alpha} \Big( U' + \frac{d-1}{r} U \Big) \Big\}'
    - 2\mu_0 (P^\alpha)' \cdot \frac{d-1}{r} U .
\end{split}
\]
Using the assumption $2\mu_0 + d \lambda_0 = 0$, we can write the right-hand side 
\[
(2\mu_0 + \lambda_0) P^\alpha 
\Big\{ \Big( U' + \frac{d-1}{r} U \Big)'
   + \alpha\frac{U' + \frac{d-1}{r}U}{\frac{1}{2}r - U} \Big( U' - \frac{1}{r} U \Big)
\Big\} .
\]
After dividing by $P^\alpha$, we have 
\begin{equation}\label{0924-1}
\begin{split}
& 
-\frac{1}{2} (r P^{1-\alpha} U) ' 
-\frac{\alpha}{2} r P^{1-\alpha}  \frac{U' + \frac{d-1}{r}U}{\frac{1}{2}r - U} U
+ \frac{1}{P^\alpha} (PU^2)' 
+ \frac{d-1}{r} P^{1-\alpha} U^2  + \frac{(PR\Theta)'}{P^\alpha}
\\
=
& 
(2\mu_0 + \lambda_0) 
\Big\{ \Big( U' + \frac{d-1}{r} U \Big)'
   + \alpha\frac{U' + \frac{d-1}{r}U}{\frac{1}{2}r - U} \Big( U' - \frac{1}{r} U \Big)
\Big\} .
\end{split}
\end{equation}
The equation \eqref{0924-1} can be written 
\begin{equation}\notag 
\begin{split}
& 
-\frac{1}{2} (r P^{1-\alpha} U) ' 
-\frac{\alpha}{2} r P^{1-\alpha}  \frac{U' + \frac{d-1}{r}U}{\frac{1}{2}r - U} U
+ \frac{1}{P^\alpha} (PU^2)' 
+ \frac{d-1}{r} P^{1-\alpha} U^2  + \frac{(PR\Theta)'}{P^\alpha}
\\
=
& 
(2\mu_0 + \lambda_0) 
\Big\{ \Big( U' + \frac{d-1}{r} U \Big)'
   + \alpha\frac{ U' - \frac{1}{r} U}{\frac{1}{2}r - U} \Big(U' + \frac{d-1}{r}U \Big)
\Big\} . 
\end{split}
\end{equation} 
Let
\begin{align}\label{def-V-tild}
\widetilde V(r) := 
\int _0 ^r \frac{U' - \frac{1}{\tilde r}U}{ \frac{1}{2} \tilde r - U}
d\tilde r .
\end{align}
Multiplying the above by $e^{\alpha \widetilde V(r)}$ we get
\begin{equation}\notag 
\begin{split}
& 
-\frac{1}{2} ( e^{\alpha \widetilde V(r)} r P^{1-\alpha} U) ' 
+ e^{\alpha\widetilde V(r)} 
 \Big\{ 
 -\frac{\alpha}{2} r P^{1-\alpha}  \frac{\frac{d}{r}U}{\frac{1}{2}r - U} U
 + \frac{1}{P^\alpha} (PU^2)' 
+ \frac{d-1}{r} P^{1-\alpha} U^2  + \frac{(PR\Theta)'}{P^\alpha}
\Big\} 
\\
=
& 
(2\mu_0 + \lambda_0) 
\Big\{ e^{\alpha \widetilde V(r)} \Big( U' + \frac{d-1}{r} U \Big)
\Big\}'  .
\end{split}
\end{equation} 
Now we integrate over $[0,r]$ 
and divide by $e^{\alpha \widetilde{V}(r)}$: 
\begin{equation}\notag 
\begin{split}
& 
-\frac{1}{2}  r P^{1-\alpha} U
\\
& 
+
e^{-\alpha \widetilde V(r)}\int_0^r 
 e^{\alpha \widetilde V(r_2)} 
 \Big\{ 
  -\frac{\alpha}{2} r_2 P^{1-\alpha}  \frac{\frac{d}{r_2}U}{\frac{1}{2}r_2 - U} U
 +\frac{1}{P^\alpha} (PU^2)' 
+ \frac{d-1}{r_2} P^{1-\alpha} U^2  + \frac{(PR\Theta)'}{P^\alpha}
\Big\} dr_2 
\\
=
& 
(2\mu_0 + \lambda_0)  \Big( U' + \frac{d-1}{r} U \Big)
- e^{-\alpha \widetilde V(r)}(2\mu_0 + \lambda_0) d A.
\end{split}
\end{equation} 
Let \begin{align} \label{def-W}
W(r):= \frac{1}{2\mu_0+\lambda_0} \int_0 ^r \frac{P(\widetilde{r})^{1-\alpha} \widetilde{r}}{2}d\widetilde{r}.
\end{align}
Multiplication by $r^{d-1} e^{W(r)} $ then yields 
\[
\begin{split}
(2\mu_0 + \lambda_0) 
\Big( r^{d-1} e^{W(r) } U \Big) ' 
=
 r^{d-1} e^{W(r)} F_U(r),
\end{split}
\]
where
\begin{align}\label{def-FU}
F_U(r) 
:=
& e^{-\alpha \widetilde V(r)} (2\mu_0 + \lambda_0 ) A 
\\
\notag &
+ e^{-\alpha \widetilde V(r)}\int_0^r 
 e^{\alpha \widetilde V(r_2)} 
 \Big\{ 
  -\frac{\alpha}{2} r_2 P^{1-\alpha}  \frac{\frac{d}{r_2}U}{\frac{1}{2}r_2 - U} U
 +\frac{1}{P^\alpha} (PU^2)' 
+ \frac{d-1}{r_2} P^{1-\alpha} U^2  + \frac{(PR\Theta)'}{P^\alpha}
\Big\} dr_2 .
\end{align}
We obtain the desired integro-differential equation after integrating the above. \\
\\
For the last equation on $\Theta$, we proceed as in \cite{GI}. Therefore we omit the details.\\
We obtain the following integro-differential formulation: 
\begin{equation}\label{0925-1}
\begin{cases}
P(r) = e^{V(r)- V(\delta)} 
\quad \text{for given } \delta > 0, 
\\
\displaystyle 
U(r) = \frac{r^{-d+1}}{2\mu_0 + \lambda _0} 
 \int_0^r r_1 ^{d-1} e^{-W(r) + W(r_1)} F_{U} (r_1)
~dr_1 , 
\\
\displaystyle 
\Theta (r) 
= (d-2) r^{-d+2} \int_0^r 
  r_1 ^{d-3} e^{-Z(r) + Z(r_1)} dr_1 \Theta _0 
  - \frac{U^2}{2C_V}
  + \frac{r^{-d+2}}{\kappa} \int_0^r r_1^{d-2} e^{-Z(r) + Z(r_1)} F_{\Theta} (r_1) dr_1 , 
\end{cases}
\end{equation}
where $V, \widetilde{V}, W$ and $F_U$ are defined as in \eqref{def-V}, \eqref{def-V-tild}, \eqref{def-W}, \eqref{def-FU}, and
\[
\begin{split}
F_\Theta(r_1) 
:= 
& 
U P \Big( \frac{U^2}{2} + C_V \Theta \Big) 
         + U P R \Theta 
+ \frac{d-2}{r_1} \int_0^{r_1}
   \Big( U P \Big( \frac{U^2}{2} + C_V \Theta \Big ) 
       + U P R \Theta 
   \Big)
   dr_2
\\
& - (2\mu_0 + \lambda_0 )
    \Big( P^\alpha UU' + \frac{d-2}{r_1} \int_0^{r_1} P^\alpha UU' ~dr_2 \Big)
\\
&  - \lambda_0 (d-1) 
    \Big( \frac{P^\alpha U^2}{r_1}
      + \frac{d-2}{r_1} \int_0^{r_1} P^\alpha\frac{U^2}{r_2}  dr_2 \Big)
\\
& + \frac{\kappa}{C_V} 
    \Big( UU' + \frac{d-2}{2r_1} U^2 \Big) ,
\end{split}
\]
and 
\[
Z(r) := \frac{C_V}{\kappa} \int_0^r \frac{P(\tilde r) \tilde r}{2 } ~d\tilde r .
\]

\subsubsection{Local existence}
In this section we construct a local solution to the above integro-differential formulation by a fixed point argument. \\
First we define the functional space in which we solve the equation: \\
Let $(1-\alpha)/2 < \varepsilon < 1-\alpha. $ Define
\[
\| (U,\Theta) \|^\delta 
:= \sup_{0<r<\delta} \Big[ 
r^{-1} |U(r)| + |U'(r)| 
+ \frac{A r( \frac{r}{\delta})^{-(1-\alpha-\varepsilon)d A} }{P_\delta^{1-\alpha} \Theta_0 } \Big| \Big( \frac{U(r)}{r} \Big)' \Big|
+ \frac{A}{\Theta_0}|\Theta(r)| + \frac{A}{\Theta_0} r^{-1} |\Theta'(r)|
\Big],
\]
\[
E^\delta = 
\{ (U,\Theta) \in \mathcal C^1 (0,\delta)  \text{ such that } 
\Theta (0) = \Theta _0 , \| (U,\Theta) \|^\delta < \infty
\}.
\]
Note that $E^{\delta}$ differs from the definition in \cite{GI}. Here we need better control of the profile $U$ near 0 to close the estimates. \\
We also define the map $\Phi$ on $B_{E^{\delta}}\big( \big(Ar,\Theta_0 \big),A/2 \big)$ by the formula 
\begin{align*}
&\Phi:(U,\Theta) \mapsto \\
& \Bigg(  \frac{r^{1-d}}{2\mu_0+\lambda_0} \int_0 ^r e^{W(r_1)-W(r)} r_1^{d-1} F_U(r_1) dr_1 , \Theta_0 + \frac{r^{2-d}}{\kappa} \int_0 ^r r_1 ^{d-2} e^{Z(r_1)-Z(r)} F_{\Theta}(r_1) dr_1 - \frac{U^2}{2 C_V} \Bigg).
\end{align*}
The remainder of the subsection is dedicated to the proof of the following lemma, which shows local existence of solutions to \eqref{0925-1}.
\begin{lem}
Assume that
\[ 
A \ll \frac{1}{2}, \qquad 
P_\delta \leq 1 , 
\qquad 
\delta +  P_\delta ^{1-\alpha} \delta^2 + \frac{P_\delta^{1-\alpha} \Theta_0}{A} 
+ \frac{A\delta}{\Theta_0} + \frac{A^2}{\Theta_0} + P_{\delta} A
\ll 1 . 
\]
Then $\Phi$ is a contraction on $B_{E^{\delta}}\big( \big(Ar,\Theta_0 \big),A/2 \big). $
\end{lem}
\begin{proof}
\underline{Stabilization}: 
First we note that
\[
|\widetilde V(r)| 
= \Big| \int_0^r \frac{\tilde r (\frac{U}{\tilde r})'}{\frac{1}{2}\tilde r - U}\Big|
\lesssim P_\delta^{1-\alpha}\Theta_0 
  \int_0^r \tilde r^{-1 } \Big(\frac{r}{\delta}\Big)^{(1-\alpha-\varepsilon)d A}d\tilde r 
\lesssim \frac{P_\delta^{1-\alpha}\Theta_0}{A}\Big(\frac{r}{\delta}\Big)^{(1-\alpha-\varepsilon)d A},
\]
which implies 
\[
|e^{\alpha \widetilde V(r)} -1 | \lesssim \frac{P_\delta^{1-\alpha}\Theta_0}{A} \ll 1 .
\]
We have the estimates 
\[
\begin{split}
P(r) 
& \lesssim  P_\delta \Big( \frac{r}{\delta} \Big) ^{d A}, 
\\
|F_U (r) - d(2\mu_0 + \lambda_0 ) A | 
& 
\lesssim 
\Big( \frac{P_\delta^{1-\alpha}\Theta_0}{A} 
+ P_\delta ^{1-\alpha} A r^2 
+  \frac{P_\delta ^{1-\alpha} \Theta _0}{A} 
+  \frac{P_\delta ^{1-\alpha}\Theta_0}{A}  r^2
\Big) A , 
\\
|F_\Theta (U) |
& 
 \lesssim P_\delta A^3 r^3 +  P_\delta A \Theta _0  r 
 + P_\delta ^\alpha A^2 r
 + A^2 r. 
\end{split}
\]
Define $(\widehat U , \widehat\Theta) := \Phi (A r , \Theta_0)$. 
We deduce that 
\[
\begin{split}
|\widehat{U} - A r|
\lesssim 
& \big( P_\delta ^{1-\alpha} r^2+ 
\frac{P_\delta^{1-\alpha}\Theta_0}{A} 
+ P_\delta ^{1-\alpha} A r^2 
+  \frac{P_\delta ^{1-\alpha} \Theta _0}{A} 
+  \frac{P_\delta ^{1-\alpha}\Theta_0}{A}  r^2\Big) A r, 
\\
|\widehat{U}' - A |
\lesssim 
& \big( P_\delta ^{1-\alpha} r^2
+ \frac{P_\delta^{1-\alpha}\Theta_0}{A} 
+ P_\delta ^{1-\alpha} A r^2 
+  \frac{P_\delta ^{1-\alpha} \Theta _0}{A} 
+  \frac{P_\delta ^{1-\alpha}\Theta_0}{A}  r^2
\Big) A , 
\\
|\widehat{\Theta} - \Theta_0| 
\lesssim 
& \Big( P_\delta \Theta _0  +A^2\Big) r^2, 
\\
|\widehat{\Theta}'| 
\lesssim 
& \Big( P_\delta \Theta _0  + A^2 \Big) r . 
\end{split}
\]
As for $(\widehat U(r)/r)'$, we write 
\[
\Big( \frac{\widehat U(r)}{r} \Big)' 
= -d \frac{U(r)}{r^2} 
+ \frac{r^{-1}}{2\mu_0 + \lambda_0} F_U(r) 
 - \frac{W'(r) }{r}U(r) .
\]
The third term satisfies 
\[
\Big| \frac{W'(r) }{r}U(r) \Big| 
\lesssim P_\delta ^{1-\alpha} \Big(\frac{r}{\delta}\Big)^{(1-\alpha)d A} A r
\lesssim \Big(\frac{A r^2}{\Theta_0} \Big) P_\delta ^{1-\alpha} \Theta_0 r^{-1} \Big(\frac{r}{\delta}\Big)^{(1-\alpha)d A}.
\]
By integration by parts, the sum of the first and the second terms is 
\[
\begin{split}
& 
-d \frac{U(r)}{r^2} + \frac{r^{-1}}{2\mu_0 + \lambda_0} F_U(r) 
\\
=
& \frac{r^{-d-1}}{2\mu_0 + \lambda_0 } \int_0^r r_1 ^{d} e^{-W(r) + W(r_1)} 
  \Big( W'(r_1)F_U(r_1) + F_U ' (r_1) \Big) ~dr_1 
\\
= 
& 
O\Bigg( \Big(\frac{A r^2}{\Theta_0} \Big) P_\delta ^{1-\alpha} \Theta_0 r^{-1} \Big(\frac{r}{\delta}\Big)^{(1-\alpha)d A}
\Bigg)   
+\frac{r^{-d-1}}{2\mu_0 + \lambda_0 } \int_0^r r_1 ^{d} e^{-W(r) + W(r_1)} 
   F_U ' (r_1)  ~dr_1 . 
\end{split}
\]
To handle the last term, we write
\[
\begin{split}
|F_U'(r) |
& 
\lesssim  | \widetilde V'| e^{-\alpha \widetilde V} A 
  + P_\delta ^{1-\alpha} A ^2  r \Big(\frac{r}{\delta}\Big)^{(1-\alpha)d A} 
  + P_\delta ^{1-\alpha} A \Theta_0r^{-1} \Big(\frac{r}{\delta}\Big)^{(1-\alpha)d A}
  + P_\delta ^{1-\alpha} \Theta_0  r \Big(\frac{r}{\delta}\Big)^{(1-\alpha)d A} 
\\
& 
\lesssim \Big( A + \frac{A^2}{\Theta_0} r^2 + r^2\Big)  
  P_\delta ^{1-\alpha} \Theta_0 r^{-1} \Big(\frac{r}{\delta}\Big)^{(1-\alpha)d A}, 
\end{split}
 \]
 and hence, 
\[
\frac{r^{-d-1}}{2\mu_0 + \lambda_0 } \int_0^r r_1 ^{d} e^{-W(r) + W(r_1)} 
   F_U ' (r_1)  ~dr_1 
\lesssim \Big( A + \frac{A^2}{\Theta_0} r^2 + r^2\Big)  
 P_\delta ^{1-\alpha} \Theta_0 r^{-1} \Big(\frac{r}{\delta}\Big)^{(1-\alpha)d A}. 
\]
We can then conclude that
\[
\begin{split}
\Big|\Big(\frac{\widehat U(r)}{r} \Big) ' \Big| 
\lesssim \Big( \frac{Ar^2}{\Theta_0} + A + \frac{A^2}{\Theta_0} r^2 + r^2\Big)  
P_\delta ^{1-\alpha} \Theta_0r^{-1} \Big(\frac{r}{\delta}\Big)^{(1-\alpha)d A}. 
\end{split}
\]
\\
\noindent
\underline{Contraction:} Let $(U_i, \Theta_i) \in B_{E^\delta} ((A r, \Theta _0), A /2), i=1,2.$ \\ Let $D := \Vert (U_1, \Theta_1) - (U_2, \Theta_2) \Vert^{\delta}.$ 
\\
Denote $\big(\widetilde{U_i}, \widetilde{\Theta_i} \big) = \Phi \big((U_i,\Theta_i) \big), i =1,2.$ \\
Using that 
\begin{align*}
\vert P_1 ^{1-\alpha} - P_2 ^{1-\alpha} \vert &  \leqslant P_{\delta}^{1-\alpha} D \ln \big( \frac{\delta}{r} \big)  \bigg( \frac{r}{\delta} \bigg)^{(1-\alpha)dA}  \\ 
                                            & \lesssim P_{\delta}^{1-\alpha} \frac{D}{A} \bigg(\frac{r}
{\delta}\bigg)^{(1-\alpha-\varepsilon)dA},
\end{align*} 
we deduce the bounds
\begin{align*}
\vert W_1 (r) - W_2 (r) \vert & \lesssim P_{\delta}^{1-\alpha} \frac{D}{A} \big( \frac{r}{\delta} \big)^{(1-\alpha-\varepsilon)dA} r^2, \  \ \ \ \vert Z_1(r) - Z_2(r) \vert \lesssim  P_{\delta} \frac{D}{A} \big( \frac{r}{\delta} \big)^{(1-\varepsilon)dA} r^2,  \\
\vert W_1 ' (r) - W_2 ' (r) \vert & \lesssim P_{\delta}^{1-\alpha} \frac{D}{A} \big(\frac{r}{\delta}\big)^{(1-\alpha-\varepsilon)dA} r,  \ \ \ \ \vert Z_1'(r)-Z_2'(r) \vert \lesssim P_{\delta} \frac{D}{A} \big( \frac{r}{\delta} \big)^{(1-\varepsilon)dA} r, \\
\vert F_{U_1}(r) - F_{U_2}(r) \vert & \lesssim \bigg( \frac{P_{\delta}^{1-\alpha} \Theta_0}{A} + P_{\delta}^{1-\alpha} A \delta^2 \bigg) D, \\
\vert F_{U_1}'(r)-F_{U_2}'(r) \vert & \lesssim \bigg(D + \frac{AD}{\Theta_0} \delta^2 + \frac{D}{A}\delta^2    \bigg)  P_{\delta}^{1-\alpha} \Theta_0 r^{-1} \big( \frac{r}{\delta} \big)^{(1-\alpha-\varepsilon)dA}, \\
\vert F_{\Theta_1}(r) - F_{\Theta_2}(r) \vert & \lesssim \big(P_{\delta}   + \frac{P_{\delta} A^2 \delta^2}{\Theta_0} + \frac{P_{\delta}^{\alpha} A}{\Theta_0} + \frac{A}{\Theta_0} \big) D r \Theta_0.
\end{align*}
Then we can estimate
\begin{align} \label{U-contr}
\notag \vert \widetilde{U_1} (r) - \widetilde{U_2} (r) \vert & \lesssim r^{1-d} \int_0 ^r r_1 ^{d-1} \vert e^{-W_1(r)+ W_1 (r_1)} - e^{-W_2(r)+ W_2 (r_1)} \vert \vert F_{U_1}(r_1) \vert dr_1 \\
\notag & + r^{1-d} \int_0 ^r r_1 ^{d-1}  e^{-W_2(r)+ W_2 (r_1)} \vert F_{U_1}(r_1) - F_{U_2}(r_1) \vert dr_1 \\
& \lesssim \bigg( \delta^2 P_{\delta}^{1-\alpha} + \frac{P_{\delta}^{1-\alpha} \Theta_0}{A} \bigg) Dr.
\end{align}
In similar fashion, we have 
\begin{align} \label{U'-contr}
\vert \widetilde{U_1} '(r) - \widetilde{U_2} '(r) \vert \lesssim \bigg( \delta^2 P_{\delta}^{1-\alpha} + \frac{P_{\delta}^{1-\alpha} \Theta_0}{A} \bigg) D.
\end{align}
Now we write
\begin{align} \label{U''-contr}
\notag \bigg \vert \bigg( \frac{\widetilde{U_1}}{r} \bigg)'-\bigg( \frac{\widetilde{U_2}}{r} \bigg)' \bigg \vert & \lesssim \frac{1}{r} \vert W_1'(r) U_1(r) - W_2'(r) U_2(r) \vert  \\
\notag &+ r^{-d-1} \int_0 ^r r_1^d \bigg \vert e^{-W_1(r)+ W_1 (r_1)} W_1'(r_1) F_{U_1}(r_1)- e^{-W_2(r)+ W_2 (r_1)} W_2'(r_1) F_{U_2}(r_1)\bigg \vert dr_1 \\
\notag &+ r^{-d-1} \int_0 ^r r_1 ^d \bigg \vert e^{-W_1(r)+ W_1 (r_1)} F_{U_1}'(r_1)- e^{-W_2(r)+ W_2 (r_1)} F_{U_2}'(r_1)\bigg \vert dr_1 \\
& \lesssim \frac{P_{\delta}^{1-\alpha} \Theta_0}{A} r^{-1} \big( \frac{r}{\delta} \big)^{(1-\alpha-\varepsilon)dA} D \Bigg( \frac{A}{\Theta_0} \delta^2 + A + \delta^2 \Bigg).
\end{align}
Similarly, we obtain the following estimates on $\Theta:$
\begin{align} \label{Thet-contr}
\vert \widetilde{\Theta_1} (r) - \widetilde{\Theta_2} (r) \vert & \lesssim \big(P_{\delta}   + \frac{P_{\delta} A^2 \delta^2}{\Theta_0} + \frac{P_{\delta}^{\alpha} A}{\Theta_0} + \frac{A}{\Theta_0} \big) D r^2 \Theta_0 + D A r^2,\\
\label{Theta'-contr} \vert \widetilde{\Theta_1}' (r) - \widetilde{\Theta_2} ' (r) \vert & \lesssim \big(P_{\delta}   + \frac{P_{\delta} A^2 \delta^2}{\Theta_0} + \frac{P_{\delta}^{\alpha} A}{\Theta_0} + \frac{A}{\Theta_0} \big) D r \Theta_0 + D A r.
\end{align}
Putting all the estimates \eqref{U-contr}, \eqref{U'-contr}, \eqref{U''-contr}, \eqref{Thet-contr} and \eqref{Theta'-contr} together, the desired result follows given our smallness assumptions.  
\end{proof}
To prove H\"{o}lder continuity of $P,$ we can repeat an argument from \cite{GI} (page 20). We omit the details here. 

\subsubsection{Global existence}
We define:
\begin{align*}
Z(r) &= \sup_{0<s<r} \frac{(1+P_{\delta}^{1/2-\alpha/2} s)^2}{M_1 s} \vert U(s) \vert + \frac{(1+P_{\delta}^{1/2-\alpha/2} s)^2}{M_1'} \big \vert U'(s) + \frac{d-1}{s}U(s) \big \vert \\
&+ \frac{(1+\sqrt{P_{\delta}}s)^2}{M_2} \Theta(s) + \frac{(1+\sqrt{P_{\delta}}s)^2}{M_2 P_{\delta} s} \vert \Theta'(s) \vert,
\end{align*}
where the constants $M_1, M_1',M_2$ are chosen so that 
\begin{align*}
M_2 \ll M_1 \ll M_1' \ll 1, A \ll M_1, \Theta_0 \ll M_2, A^2 \ll P_{\delta} M_2, M_1' \log \frac{1}{\delta^2 P_{\delta}^{1-\alpha}} \leqslant 1, \\
\frac{P_\delta^{1-\alpha}}{A}M_1 ' \leq 1 ,
 M_2 \ll M_1 P_{\delta}^{1/2+\alpha}, M_1 M_1' \ll P_{\delta} M_2, M_1^3 \ll P_{\delta}^{1-2\alpha} \Theta_0, M_1 M_1 ' \ll \Theta_0 P_{\delta}^{1-\alpha}.
\end{align*}
We make use of $Z$ and a bootstrap argument to prove global existence. This is contained in the following lemma:
\begin{lem}
We have $Z(\delta) \leqslant \frac{1}{2}.$ \\
Moreover if $Z(r) \leqslant 1$ for some $r> \delta,$ then the stronger bound $Z(r) \leqslant \frac{1}{2}$ holds.
\end{lem}
\begin{proof}
\underline{Starting point:} \\
To ensure $Z(\delta) \ll 1,$ we require, as in \cite{GI}:
\begin{align*}
A \ll M_1, \Theta_0 \ll M_2, A^2 \ll P_{\delta} M_2.
\end{align*}
\underline{Estimate on $\widetilde{V}:$} We note that for $r>\delta:$
\begin{align*}
\vert \widetilde{V}(r) \vert & \lesssim \int_0 ^\delta P_{\delta}^{1-\alpha}\Theta_0 r_1 ^{-1} \big( \frac{r_1}{\delta} \big)^{(1-\alpha - \varepsilon)dA} dr_1 + \int_{\delta}^{r} \frac{M_1'+M_1}{r_1 \big(1+ P_{\delta}^{1/2-\alpha/2} r_1 \big)^2} ds \\ & \lesssim \frac{P_{\delta}^{1-\alpha} \Theta_0}{A} + (M_1+M_1') \big[ \log \frac{1}{P_{\delta}^{1-\alpha} \delta^2} + 1 \big] \lesssim 1,
\end{align*}
provided $(M_1'+M_1) \log \frac{1}{\delta^2 P_{\delta}^{1-\alpha}} = O (1).$
\\
By a similar reasoning, we can show that under that same assumption, 
\begin{align*}
P_{\delta} \lesssim P(r) \lesssim P_{\delta},
\end{align*}
for $r > \delta.$ \\
\\
\underline{Estimates on $U$}: We begin with bounds on $F_U.$ \\ 
Using the fact that $P(r) \leqslant P_{\delta} \big( \frac{r}{A} \big)^{dA}$ for $r \in [0, \delta],$ we write
\begin{align*}
\int_0 ^{r} e^{\alpha \widetilde{V}(r_1)} \frac{(PR\Theta)'}{P^{\alpha}} dr_1 \lesssim \frac{P_{\delta}^{1-\alpha}}{A} M_1' M_2 + P_{\delta}^{1-\alpha} M_2 \delta^2 
+ \Bigg \vert \int_{\delta} ^r e^{\alpha \widetilde{V}(r_1)} \frac{(PR\Theta)'}{P^{\alpha}} dr_1 \Bigg \vert .
\end{align*} 
After integrating by parts, we obtain
\begin{align*}
\Bigg \vert \int_{\delta} ^r e^{\alpha \widetilde{V}(r_1)} \frac{(PR\Theta)'}{P^{\alpha}} dr_1 \Bigg \vert
\lesssim P_{\delta}^{1-\alpha} \Theta_0 + P_\delta^{1-\alpha} M_2 
+ M_2(M_1+M_1') \log \frac{1}{P_{\delta} \delta^2} 
 \lesssim P_\delta^{1-\alpha} \Theta_0 + M_2.
\end{align*}
We deduce from the above that
\begin{align} 
\label{F_U} \vert F_U(r) \vert & \lesssim A +  M_1 M_1' + P_{\delta}^{1-\alpha} \Theta_0  + M_2 \\
\label{F_U'}\vert F_{U}'(r) \vert & \lesssim\frac{ M_1' A +  M_1 M_1' +  P_{\delta}^{1-\alpha} \Theta_0  + M_2}{r}
\end{align}
Together with the fact that
\begin{align*}
r^{1-d} \int_0 ^r r_1^{d-1} e^{-C P_{\delta}^{1-\alpha} (r^2-r_1^2)} dr_1 \lesssim \frac{r}{\big(1+P_{\delta}^{1/2-\alpha/2}r \big)^2},
\end{align*}
\eqref{F_U} implies
\begin{align*}
\vert U(r) \vert \lesssim \frac{r}{\big(1+P_{\delta}^{1/2-\alpha/2}r \big)^2} 
\bigg(  A +  M_1 M_1' + P_{\delta}^{1-\alpha} \Theta_0 + M_2 
\bigg).
\end{align*}
Moving on to the derivative part, we first deal with $r \leqslant \frac{1}{P_{\delta}^{1/2-\alpha/2}}.$ We write that  
\begin{align*}
\bigg \vert U' + \frac{d-1}{r} U \bigg \vert & = \bigg \vert \frac{1}{2\mu_0+\lambda_0} F_U(r) - W'(r) U(r) \bigg \vert \\
                                           & \lesssim A + M_1 M_1' + P_{\delta} \Theta_0 + \frac{M_2}{P_{\delta}^{1/2+\alpha}} +  M_1.
\end{align*}
When $r \geqslant \frac{1}{P_{\delta}^{1/2-\alpha/2}},$ we write, using an integration by parts:
\begin{align} \label{IPPder}
\bigg \vert U' + \frac{d-1}{r} U \bigg \vert & = \Bigg \vert W'(r) \frac{r^{1-d}}{2\mu_0+\lambda_0} \int_0 ^r e^{-W(r)+W(s)} \partial_{s} \bigg( \frac{ s^{d-1} F_U(s)}{W'(s)} \bigg) ds \Bigg \vert.
\end{align}
We have by \eqref{F_U'} 
\begin{align*}
\Bigg \vert  \partial_{s} \bigg( \frac{ s^{d-1} F_U(s)}{W'(s)} \bigg) \Bigg \vert \lesssim r^{d-3} \frac{ A +  M_1 M_1' + P_{\delta}^{1-\alpha} \Theta_0 + M_2  }{P_{\delta}^{1-\alpha}} \bigg( 1 + \big(\frac{r}{\delta} \big)^{-4dA(1-\alpha)} \bigg) .
\end{align*}
Together with \eqref{IPPder} and the fact that 
\begin{align*}
r^{1-d} \int_0 ^r r_1 ^{d-3} e^{-C P_{\delta}^{1-\alpha}(r^2-r_1^2)} dr_1 \lesssim \frac{1}{P_{\delta}^{1-\alpha} r^3} ~ \textrm{when} ~ r \geqslant \frac{1}{P_{\delta}^{1/2-\alpha/2}},
\end{align*}
we conclude that 
\begin{align*}
\bigg \vert U' + \frac{d-1}{r} U \bigg \vert & \lesssim P_{\delta}^{1-\alpha} r \frac{1}{P_{\delta}^{1-\alpha}r^3} \frac{A +  M_1 M_1' + P_{\delta}^{1-\alpha} \Theta_0 + M_2}{P_{\delta}^{1-\alpha}} \\
                                             & \lesssim \frac{1}{P_{\delta}^{1-\alpha}r^2} \bigg( A +  M_1 M_1' + P_{\delta}^{1-\alpha} \Theta_0 + M_2 \bigg).
\end{align*}
\underline{Estimates on $\Theta$:} \\
We start with the following estimates
\begin{align}
\label{F-theta-s}\vert F_{\Theta}(r) \vert & \lesssim  M_1 M'_1 r  , ~\textrm{when}~  r \leqslant \frac{1}{\sqrt{P_{\delta}}},\\
\label{F-theta-l}\vert F_{\Theta}(r) \vert & \lesssim \frac{1}{r} \bigg( \frac{M_1^3}{P_{\delta}^{1-2\alpha}} +   \frac{M_1 M_1'}{P_{\delta}^{1-\alpha}} \bigg) , ~\textrm{when}~  r \geqslant \frac{1}{\sqrt{P_{\delta}}}.
\end{align}
We deduce from \eqref{F-theta-s} that
\begin{align*}
\vert \Theta(r) \vert \lesssim \Theta_0 + \frac{M_1 M'_1}{P_{\delta}}   .
\end{align*}
Now for the case of $r$ large, using that 
\begin{align*}
r^{2-d} \int_0 ^r r_1 ^{d-3} e^{-C P_{\delta}(r^2-r_1^2)} dr_1 \lesssim \frac{1}{(1+\sqrt{P_{\delta}}r)^2},
\end{align*}
we deduce from \eqref{F-theta-l} that
\begin{align*}
\vert \Theta(r) \vert \lesssim \Big( \frac{1}{\sqrt{P_{\delta}} r} \Big) ^2 \big( \Theta_0 +  \frac{M_1^3}{P_{\delta}^{1-2\alpha}} + \frac{M_1M_1'}{P_{\delta}^{1-\alpha}} \big).
\end{align*}
Similarly for the derivative, we use  
\begin{align*}
\Bigg \vert \partial_r \Bigg[ r^{2-d} \int_0 ^r r_1 ^{d-3} e^{-C P_{\delta}(r^2-r_1^2)} \Bigg] dr_1 \Bigg \vert \lesssim \frac{P_{\delta}r}{1+P_{\delta}r^2},
\end{align*}
and deduce 
\begin{align*}
\vert \Theta'(r) \vert & \lesssim P_{\delta}r \Bigg(\Theta_0 + \frac{M_1^3}{P_{\delta}^{1-2\alpha}} +\frac{M_1M_1'}{P_{\delta}^{1-\alpha}} \Bigg) , ~\textrm{when}~ r \leqslant \frac{1}{\sqrt{P_{\delta}}},   \\
\vert \Theta'(r) \vert & \lesssim \frac{1}{r} \Bigg( \Theta_0+\frac{M_1^3}{P_{\delta}^{1-2\alpha}} + \frac{M_1M_1'}{P_{\delta}^{1-\alpha}} \Bigg)  , ~\textrm{when}~ r \geqslant \frac{1}{\sqrt{P_{\delta}}}.
\end{align*}
\end{proof}
Global existence directly follows from the previous lemma by a standard continuation argument. 
\subsubsection{Asymptotic behavior}
Regarding the asymptotic behavior of $P, U$ and $\Theta,$ the arguments are almost identical to those of \cite{GI}. Therefore we only state the condition needed on $M_1,M_1',M_2$ to ensure positivity of $\Theta:$ 
$$
\frac{M_1^3}{P_{\delta}^{1-2\alpha}} + \frac{M_1 M_1'}{P_{\delta}^{1-\alpha}} \ll \Theta_0.
$$
\subsection{Existence when $\alpha = 1$}
In this section, we show an existence theorem in the case where $\alpha = 1.$ Unlike in the previous section, we now assume $2\mu_0+d\lambda_0 >0.$  \\
We show the following existence result:  
\begin{thm}\label{frwd>0}
Let $d \geqslant 3.$ Fix $(C_V, \kappa, R, \mu_0, \lambda_0) \in (0,\infty)^{5}$ such that $2\mu_0 + d \lambda_0 >0.$ \\
Then there exists a constant $C(C_V, \kappa, R, \mu_0, \lambda_0) := C>0$ such that if 
\begin{align*}
 A \log \delta ^{-1} + P_{\delta} + \delta <C ,
\end{align*}
there exists a solution $(P,U,\Theta) \in \mathcal{C}^{\frac{dA}{1/2-A}} \times \mathcal{C}^1 ([0,\infty)) \times  \mathcal{C}^1([0,\infty))$ to \eqref{ODEs:fwrd} such that $P(0)=0, P(\delta) = P_{\delta}, U(0)=0, U'(0)=A, \Theta(0)=\Theta_0:=\frac{1}{R}(2\mu_0+d\lambda_0)A, \Theta'(0)=0.$ \\
\\
Moreover the profiles satisfy the following global bounds:
\begin{align*}
P(r) & \simeq P_{\delta} \min \Bigg[1, \bigg( \frac{r}{\delta} \bigg)^{\frac{2dA}{1-2A}} \Bigg],\\
\vert U(r) \vert \lesssim \frac{Ar}{(1+r)^2}, & \ \ \ \ \  \vert U'(r) \vert \lesssim \frac{A}{(1+r)^2}, \\
0 \leqslant \Theta(r) \lesssim \frac{1}{(1+\sqrt{P_{\delta}}r)^2},& \ \ \ \ \ \vert \Theta'(r) \vert \lesssim \frac{\sqrt{P_{\delta}}r}{(1+\sqrt{P_{\delta}}r)^2}.
\end{align*}
Finally, there exist $P_{\infty} >0, U_{\infty}>0, \Theta_{\infty} >0$ such that as $r \to +\infty:$
\begin{align*}
P(r) &= P_{\infty} + O \big( \frac{1}{r^2}\big), \\
U(r) &= \frac{U_{\infty}}{r} +  O \big( \frac{1}{r^3}\big), \\
\Theta(r) &= \frac{\Theta_{\infty}}{r^2} + O \big( \frac{1}{r^4}\big).
\end{align*}
\end{thm}
\begin{rem}
Note that compared to Theorem \ref{frwd-0}, the choice of $\Theta_0$ is not free.
\end{rem}
\begin{proof}
We proof is similar to that of Theorem \ref{frwd-0} above. Therefore we only outline the parts of the argument that differ, namely the derivation of the integro-differential equation, and the proof of local existence. \\
\\
\textbf{Integro-differential formulation:} \\
We consider the following initial conditions:
\begin{align*}
P(0)=0, ~~ U(0)=0, ~~ U'(0)= A, ~~ \Theta(0) = \Theta_0:= \frac{2\mu_0+d\lambda_0}{R}A, ~~ \Theta'(0)=0.
\end{align*}
First, we integrate the equation on $P,$ and obtain for some $\delta>0:$
\begin{align}\label{expP}
P(r) = e^{V(r)-V(\delta)} P(\delta),
\end{align}
where
\begin{align*}
V(r) - V(\delta) = \int_{\delta} ^r \frac{U' + \frac{d-1}{r_1}U}{\frac{1}{2}r_1-U} dr_1.
\end{align*}
The equation on $U$ reads:
\begin{align*}
& -\frac{1}{2} (r PU)' + (PU^2)' + \frac{d-1}{r} PU^2 - (2\mu_0+\lambda_0)P \bigg(U' + \frac{d-1}{r}U   \bigg)' + (PR\Theta)'  \\
&= (2\mu_0+\lambda_0) P'U' + \lambda_0  P' \frac{d-1}{r}U.
\end{align*}
We integrate the previous expression, perform an integration by parts and divide by $P$ to obtain:
\begin{align*}
&-\frac{1}{2}rU +  U^2 + P^{-1} \int_0 ^r \frac{d-1}{r_1} PU^2 dr_1 - (2\mu_0+\lambda_0) \big(U' + \frac{d-1}{r} U \big) + R \Theta \\
&=-2\mu_0 P^{-1} \int_0 ^r P' \frac{d-1}{r_1} U dr_1.
\end{align*}
The necessity of the relation between $A$ and $\Theta_0$ can be seen in this equation. Since we expect the behavior $U(r) \simeq A r$ near 0, equating the higher order terms on each side yields
\[
- (2\mu_0+\lambda_0) \big(U' + \frac{d-1}{r} U \big) + R \Theta 
+2\mu_0 P^{-1} \int_0 ^r P' \frac{d-1}{r_1} U dr_1
\simeq -(2\mu _0 + d\lambda_0) A + R \Theta_0 .
\]
This directly gives
$-(2\mu _0 + d\lambda_0) A + R \Theta_0=0$. \\
Now we multiply by $r^{d-1} e^{W(r)}$ (where $W(r)= \frac{r^2}{4(2\mu_0+\lambda_0)}$) and obtain
\begin{align*}
-(2\mu_0+\lambda_0) \bigg(r^{d-1} e^{W(r)} U \bigg)' &=- r^{d-1} e^{W(r)} \bigg(R \Theta + U^2  + P^{-1} \int_0 ^r \frac{d-1}{r_1} PU^2 dr_1\\
&+2\mu_0 P^{-1} \int_0 ^r P' \frac{d-1}{r_1}U dr_1
\bigg).
\end{align*}
We conclude with the integro-differential satisfied by $U:$
\begin{align*}
U(r) &= \frac{r^{1-d}}{2\mu_0+\lambda_0} \int_0 ^r e^{W(r_1)-W(r)} r_1^{d-1} F_U(r_1) dr_1,
\end{align*}
where
\begin{align*}
F_U(r) &=  R \Theta + U^2 + P^{-1} \int_0 ^r \frac{d-1}{r_1} PU^2 dr_1+ 2\mu_0 P^{-1} \int_0 ^r P' \frac{d-1}{r_1}U dr_1. 
\end{align*}
For $\Theta,$ we do not detail the argument since it is similar to the above. We get
\begin{align*}
\Theta(r) = (d-2) r^{2-d} \int_0 ^r r_1 ^{d-3} e^{-Z(r)+Z(r_1)} dr_1 \Theta_0 + \frac{r^{2-d}}{\kappa} \int_0 ^r r_1 ^{d-2} e^{Z(r_1)-Z(r)} F_{\Theta}(r_1) dr_1 - \frac{U^2}{2 C_V} 
,
\end{align*}
where
\begin{align*}
Z(r) & := \frac{C_V}{2\kappa} \int_0 ^r \tilde{r} P(\tilde{r}) ~d\tilde{r} , \\
F_{\Theta}(r) &:=  UP \bigg(\frac{U^2}{2} + C_V \Theta \bigg) + UPR \Theta + \frac{d-2}{r} \int_0 ^r \bigg(UP \bigg( \frac{U^2}{2} + C_V \Theta \bigg) + UPR\Theta \bigg) dr_2 \\
&+ \big(\frac{\kappa}{C_V}-(2\mu_0+\lambda_0)\big) \frac{(U^2)'}{2} + \frac{\kappa}{C_V} \frac{d-2}{2r} U^2\\
 &+\frac{2\mu_0+\lambda_0}{r} \int_0 ^r \bigg((2-d) r_1 ^{-d+1} P + P r_1^{-d+2} \frac{U' + \frac{d-1}{r}U}{\frac{1}{2}r - U} \bigg)(r_1^{d-1} UU') dr_1 \\
&-\lambda \frac{d-1}{r} U^2 +\lambda_0 \frac{d-1}{r} \int_0 ^r \big(r_1^{-d+2} P \big)'(r_1^{d-2}U^2)dr_1 \\
&+\frac{(2\mu_0+\lambda_0)}{r} \int_0 ^r r_1 P' UU' dr_1 +\frac{(d-1)\lambda_0}{r} \int_0 ^r P' U^2 dr_1
\end{align*}
\textbf{Local existence:} \\
Fix $\varepsilon>0$ such that $A \ll \varepsilon \ll 1.$ Assume that $\delta$ satisfies $\delta^{\varepsilon}/\varepsilon \ll 1.$ We define the following norm (slightly different from the above): 
\begin{align*}
\Vert (U,\Theta) \Vert^{\delta} := \sup_{0<r<\delta} \bigg[ r^{-1} \vert U(r) \vert + 
r^{1-\varepsilon} \big \vert \big( \frac{U}{r} \big)' \big \vert + \Theta(r) + r^{-1} \vert \Theta'(r) \vert \bigg].
\end{align*}
We apply the fixed point argument with the above norm in the space 
\[
B_{E^\delta}( (Ar, \Theta_0) , A/2) \cap 
\Big\{ \lim_{r\to 0} \frac{U(r)}{r} = A , 
       \lim_{r\to 0} \Theta(r) = \Theta_ 0
 \Big\} .
\]
\\
\noindent \underline{Stabilization:} We have the following estimates:
\begin{align*}
\vert F_U(r) - d(2\mu_0+\lambda_0)A \vert & \lesssim A \bigg( \delta^2 
+ \frac{\delta^{\varepsilon}}{\varepsilon} \bigg), \\ 
\vert F_{\Theta}(r) \vert \lesssim (A^2 + P_{\delta}A) r ,
\end{align*}
where we used that $\Theta_0 = \frac{1}{R}(2\mu_0 + d\lambda_0)A.$ \\
Denote $(\widehat{U},\widehat{\Theta})=\Phi \big((Ar,\Theta_0) \big)$. \\
We can deduce from the above that
\begin{align*}
\vert \widehat{U}-Ar \vert & \lesssim A r \bigg(  \delta^2 + \frac{\delta^{\varepsilon}}{\varepsilon} \bigg) , \\
\vert \widehat{\Theta} - \Theta_0 \vert & \lesssim (A^2 + P_{\delta} A)r^2, \\
\vert \widehat{\Theta}' \vert &  \lesssim (A^2 + P_{\delta} A)r .
\end{align*}
For the second $U-$part of the norm we can write, using that $P'>0$ on $(0,\delta)$:
\begin{align*}
\vert F_U'(r) \vert &  \lesssim \vert \Theta' \vert + \vert UU' \vert + \bigg \vert \bigg( P^{-1} \int_0 ^r P \frac{U^2}{r_1} dr_1 \bigg)' \bigg \vert + \frac{A}{r} P^{-1} \int_0 ^r P' \bigg \vert \frac{(d-1)U(r_1)}{r_1} - \frac{(d-1)U(r)}{r} \bigg \vert dr_1 \\
                    & \lesssim Ar +  \frac{A^2}{\varepsilon} r^{\varepsilon-1}.
\end{align*}
We deduce the following estimate as we did in Section \ref{frwd-0}:
\begin{align*}
\big \vert \big( \frac{\widehat{U}}{r} \big)' \big \vert \lesssim A r^{-1+\varepsilon} \big( \delta^{2-\varepsilon} +  \frac{A}{\varepsilon} \big).
\end{align*}
\underline{Contraction:}
Let $(U_i,\Theta_i) \in B_{E^{\delta}}\big((Ar,\Theta_0), A/2 \big).$ Denote $D:= \Vert (U_1,\Theta_1) - (U_2,\Theta_2) \Vert^{\delta}.$ \\
Finally, let $(\widetilde{U_i}, \widetilde{\Theta_i}):=\Phi\big( (U_i,\Theta_i) \big).$ \\
We have, arguing as in the stabilization part of the proof:
\begin{align*}
\vert \widetilde{U_1}(r) - \widetilde{U_2}(r) \vert & \lesssim D r \bigg(  \delta^2 + \frac{\delta^{\varepsilon}}{\varepsilon} \bigg), \\
\big \vert \big( \frac{\widetilde{U_1}}{r} \big)' - \big( \frac{\widetilde{U_2}}{r} \big)' \big \vert & \lesssim D r^{-1+\varepsilon} \big( \delta^{2-\varepsilon} +  \frac{A^2}{\varepsilon} \big), \\
\vert \Theta_1(r) - \Theta_2(r) \vert & \lesssim  A D r^2 ,\\ 
\vert \Theta_1 '(r) - \Theta_2 ' (r) \vert & \lesssim A D r .
\end{align*}
The proofs of global existence and asymptotic behavior of the solutions are done similarly as in Section \ref{frwd-0}, therefore details are omitted.
\end{proof}

\section{Non-existence of shrinkers} \label{backwrd}
We now investigate shrinker-type solutions to the system \eqref{eq:cNS}, that is solutions of the type \eqref{shrinkers}. \\
Plugging the corresponding ansatz into the system, we obtain the following ODEs satisfied by the profiles $P,U$ and $\Theta$: \\ 
For $r = |x| > 0$ 
\begin{equation}\label{eq:ODEs_back}
\begin{cases}
\displaystyle 
\frac{1}{2} r P ' + P' U 
+ P \Big( U' + \frac{d-1}{r} U \Big) 
= 0 , 
\\[5mm]
\begin{split}
\displaystyle 
\frac{1}{2} 
 P  U 
& 
+ \frac{1}{2} r (P U)' 
+ (P U ^2) ' + \frac{d-1}{r} P U^2 
+(P R \Theta)'
\\
= 
& 
(2\mu_0 + \lambda_0) P^{\alpha} 
 \Big( U'' + \frac{d-1}{r} U' - \frac{d-1}{r^2} U \Big) 
 \\&- (2 \mu_0 + \lambda_0)\alpha \frac{U' +  \frac{d-1}{r} U}{\frac{1}{2}r + U} P^{\alpha} U' - \lambda_0 \alpha \frac{U' + \frac{d-1}{r} U}{\frac{1}{2}r + U} P^{\alpha} \frac{d-1}{r} U ,
\end{split}
\\[10mm]
\displaystyle 
\begin{split}
P 
 \Big( \frac{U^2}{2}
&  + C_V \Theta \Big) 
+ \frac{1}{2} r 
  \Big( P \Big( \frac{U^2}{2} + C_V \Theta \Big) \Big) '
+ \Big( U P \Big( \frac{U^2}{2} + C_V \Theta \Big) 
        + U P R \Theta \Big) '
\\
&
+ \frac{d-1}{r} 
  \Big( U P \Big( \frac{U^2}{2} + C_V \Theta \Big)
       + U P R \Theta
  \Big)
- \kappa \Big( \Theta '' + \frac{d-1}{r} \Theta ' \Big)
\\
= 
& 
2\mu_0 P^{\alpha} \Big( (U')^2 + \frac{d-1}{r^2} U^2 \Big) 
+ \lambda_0 P^{\alpha} \Big( U' + \frac{d-1}{r} U \Big) ^2
\\
& 
+ (2\mu_0 + \lambda_0) P^{\alpha}
 \Big( U'' + \frac{d-1}{r}U' - \frac{d-1}{r^2} U \Big) U
\\
&
- (2 \mu_0 + \lambda_0)\alpha \frac{U' +  \frac{d-1}{r} U}{\frac{1}{2}r + U} P^{\alpha} U'U - \lambda_0 \alpha\frac{U' +  \frac{d-1}{r} U}{\frac{1}{2}r + U}P^{\alpha} \frac{d-1}{r} U^2.  
\end{split}
\end{cases}
\end{equation}
\noindent The main result of this section is that \eqref{eq:ODEs_back} does not have non trivial small solutions. 
\\
As mentioned above, given the dependence laws that we have adopted for the Lam\'{e} coefficients, it is natural to restrict to solutions that exhibit cavitation. More precisely, we will make the following mild assumptions on the density profile $P$ throughout this section:
\begin{align} \label{cavitation}
\bullet ~~ & \exists \varepsilon, P_{\varepsilon} >0, \forall r \geqslant \varepsilon, P(r) \geqslant P_{\varepsilon}, \\
\notag \bullet~~ & \exists \Lambda >0,~ \forall r \in (0,\varepsilon),~ \frac{r \int_0 ^r P^{1-\alpha}(r_1) dr_1}{\int_0 ^r P^{1-\alpha}(r_1) r_1 dr_1} , \frac{\int_0 ^r P}{P^{\alpha} \int_0^r P^{1-\alpha}} \leqslant \Lambda.
\end{align}
Note that the functions $P$ constructed in Section \ref{fwrd} satisfy this condition. Since the forward and backward systems are essentially equivalent near the origin, a shrinker exhibiting cavitation would be expected to behave similarly near 0.
\\
The main theorem of this section is:
\begin{thm}\label{main}
Let $d \geqslant 3$ and $0<\alpha<1.$ Fix $(C_V, \kappa, R, \mu_0, \lambda_0) \in (0,\infty)^5$ such that $C_V \leqslant \frac{\kappa}{2\mu_0+\lambda_0}.$ \\
Assume that the density function $P$ satisfies the above condition \eqref{cavitation}. \\
Then there exists a constant $C(C_V,\kappa,R, \mu_0,\lambda_0,\varepsilon, P_{\varepsilon},\Lambda):=C>0$ such that if
\begin{align*}
\sup_{r>0} \bigg( \langle r \rangle^2 \Theta + P^{1-\alpha} + \bigg \vert \frac{U}{r \Theta} \bigg \vert \bigg)  + \sup_{r>\varepsilon} \bigg \vert \frac{U'}{r \Theta'} \bigg \vert < C , \\
\end{align*}
then $U \equiv \Theta \equiv 0, P = Constant.$

\end{thm}
\begin{rem}
The assumption $d \geqslant 3$ comes from the fact that we use Hardy's inequality in the proof.
\end{rem}
\begin{rem}
Note that by scaling we expect the behavior $U(r) \sim \displaystyle \frac{U_{\infty}}{r}$ and $\Theta \sim \displaystyle \frac{\Theta_{\infty}}{r^2}$ at infinity. This makes the assumptions in the theorem critical.
\end{rem}
\begin{rem}
The proof can be adapted with minor changes when $\alpha = 1,$ or when the Lam\'{e} coefficients are not density dependent. For example this method could be applied in the case where $C_V \geqslant \frac{\kappa}{2\mu + \lambda}.$ In this case the smallness condition would be written in terms of $\sup_{r>\varepsilon} \big \vert \frac{\Theta}{r U} \big \vert $ and $\sup_{r>\varepsilon} \big \vert \frac{\Theta'}{r U'} \big \vert.$ 
\end{rem}
\begin{rem}\label{beta}
Note that our smallness assumptions imply that there exists $0 < b \ll 1 $ such that
\begin{align*}
\sup_{r>0} \bigg \vert \frac{U}{r} \bigg \vert \leqslant b.
\end{align*}
\end{rem}
\begin{rem}\label{C_V}
We notice that for any $r >0,$ $C_V \leqslant \frac{1}{2} \frac{\kappa}{2\mu_0 + \lambda_0} \frac{1}{P^{\alpha}},$
given the smallness assumption on $P.$
\end{rem}
\noindent The remainder of this section is dedicated to the proof of Theorem \ref{main}.

\subsection{Set-up}
Let us recall $\mu = \mu_0 P ^\alpha ,\lambda = \lambda_0 P^\alpha $. We start by writing the equations on $\Theta$ and $U$ from \eqref{eq:ODEs_back} in $2 \times 2$ matrix form: 
\begin{eqnarray*}&& \left( \begin{array}{ll}
C_V P \big(\frac{r}{2} + U \big) \displaystyle \frac{d}{dr}   &    P R \Theta \displaystyle \frac{d}{dr}  \\
               P R  \displaystyle  \frac{d}{dr}              &   P \big(\displaystyle\frac{r}{2} + U \big) \displaystyle \frac{d}{dr} 
\end{array} \right) 
\left( \begin{array}{ll}
\displaystyle \Theta \\
\displaystyle U
\end{array} \right) \\
&-& \left( \begin{array}{cc}
\kappa \big(\displaystyle \frac{d-1}{r} \frac{d}{dr} -\displaystyle \frac{d^2}{dr^2} \big)   &    0  \\
0                    &    (2\mu+\lambda) \big( \displaystyle \frac{d-1}{r} \frac{d}{dr} -\displaystyle \frac{d^2}{dr^2}  \big)
\end{array}
\right) \left( \begin{array}{ll}
\Theta \\
U
\end{array} \right) \\
&=& \left( \begin{array}{ll}
-\displaystyle P(r) \Theta (r) \big( C_V + R \frac{d-1}{r}U \big) + 2 \mu \bigg( (U')^2 + \big(\displaystyle \frac{d-1}{r} U \big)^2 \bigg) + \lambda \bigg(U' + \frac{d-1}{r} U \bigg)^2 
\\
-\displaystyle \frac{1}{2} P(r) U(r) - (2\mu+\lambda)\displaystyle \frac{d-1}{r^2} U
\end{array}
\right)
\\
&+& 
\left( \begin{array}{ll}
 - (2 \mu + \lambda)\alpha \displaystyle \frac{ U' +  \frac{d-1}{r} U}{\frac{1}{2}r + U} U'U - \lambda \alpha \displaystyle \frac{U' +  \frac{d-1}{r} U}{\frac{1}{2}r + U}\displaystyle\frac{d-1}{r} U^2
\\
- (2 \mu + \lambda)\alpha \displaystyle \frac{U' +  \frac{d-1}{r} U}{\frac{1}{2}r + U} U' - \lambda \alpha \displaystyle \frac{U' +  \frac{d-1}{r} U}{\frac{1}{2}r + U} \displaystyle \frac{d-1}{r} U
\end{array}
\right).
\end{eqnarray*}
We denote 
\begin{eqnarray*}
B &=& \left( \begin{array}{cc}
\kappa  &    0   \\
0       &   2\mu + \lambda
\end{array}
\right), \\
\widetilde{A}(r) &=&  \left( \begin{array}{ll}
C_V P \big(\frac{r}{2} + U \big)   &    P R \Theta  \\
               P R                 &   P \big(\frac{r}{2} + U \big) 
\end{array}
\right).
\end{eqnarray*}
The previous equation can then be written
\begin{eqnarray*}
&&\widetilde{A}(r)  \frac{d}{dr}\left( \begin{array}{ll}
\Theta \\
U
\end{array} \right) - B \big(\displaystyle \frac{d-1}{r} -\displaystyle \frac{d^2}{dr^2} \big) 
\left( \begin{array}{ll}
\Theta \\
U
\end{array} \right) \\
&=& \left( \begin{array}{ll}
- P(r) \Theta (r) \big( C_V + R \frac{d-1}{r}U \big) + 2 \mu \bigg( (U')^2 + \big( \displaystyle \frac{d-1}{r} U \big)^2 \bigg) + \lambda \bigg(U' +\displaystyle \frac{d-1}{r} U \bigg)^2
\\
-\displaystyle \frac{1}{2} P(r) U(r) - (2\mu+\lambda)\displaystyle \frac{d-1}{r^2} U
\end{array}
\right) \\
\\
&+& 
\left( \begin{array}{ll}
 - (2 \mu + \lambda)\alpha \displaystyle \frac{U' +  \frac{d-1}{r} U}{\frac{1}{2}r + U} U'U - \lambda \alpha\displaystyle \frac{U' +  \frac{d-1}{r} U}{\frac{1}{2}r + U} \displaystyle \frac{d-1}{r} U^2
\\
- (2 \mu + \lambda)\alpha \displaystyle \frac{U' +  \frac{d-1}{r} U}{\frac{1}{2}r + U} U' - \lambda \alpha \displaystyle \frac{U' + \frac{d-1}{r} U}{\frac{1}{2}r + U} \displaystyle \frac{d-1}{r} U
\end{array}
\right).
\end{eqnarray*}
Hence after multiplication by $B^{-1}$ and $ \exp \big(- \int_0 ^r B^{-1} \widetilde{A}(r') dr \big)$, we obtain:
\begin{eqnarray}\label{eq}
&& - \nabla \cdot \bigg( \exp \big(- \int_0 ^r B^{-1} \widetilde{A}(r') dr \big) \nabla \left( 
\begin{array}{ll}
\Theta \\
U
\end{array}
\right) \bigg) \\
\notag &=& \exp \big(- \int_0 ^r B^{-1} \widetilde{A}(r') dr \big) \\
\notag & \times & \left( \begin{array}{ll}
-\displaystyle \frac{P(r) \Theta (r)}{\kappa} \big(C_V + R \frac{d-1}{r}U \big) + 2 \displaystyle \frac{\mu}{\kappa} \bigg( (U')^2 + \big(\displaystyle \frac{d-1}{r} U \big)^2 \bigg) + \displaystyle \frac{\lambda}{\kappa} \bigg(U' + \displaystyle \frac{d-1}{r} U \bigg)^2
\\
\notag- \displaystyle \frac{1}{2(2\mu+\lambda)} P(r) U(r) - \displaystyle \frac{d-1}{r^2} U
\end{array}
\right)
\\
\notag &+&  \exp \big(- \int_0 ^r B^{-1} \widetilde{A}(r') dr \big) \\
\notag & \times & \left( \begin{array}{ll}
 -\displaystyle \frac{(2 \mu + \lambda)\alpha}{\kappa} \displaystyle \frac{U' + \frac{d-1}{r} U}{\frac{1}{2}r + U} U'U -\displaystyle \frac{\lambda}{\kappa} \alpha \displaystyle \frac{U' +  \frac{d-1}{r} U}{\frac{1}{2}r + U} \displaystyle \frac{d-1}{r} U^2
\\
-\alpha \displaystyle \frac{U' +  \frac{d-1}{r} U}{\frac{1}{2}r + U} U' -\displaystyle \frac{\lambda}{2\mu+\lambda} \alpha \displaystyle \frac{U' + \frac{d-1}{r} U}{\frac{1}{2}r + U} \displaystyle \frac{d-1}{r} U
\end{array}
\right).
\end{eqnarray}
Denote from now on 
\begin{eqnarray*}
\exp(A(r)) &:=&  \exp \big(- \int_0 ^r B^{-1} \widetilde{A}(r') dr \big) \\
           & := & \exp \left( \begin{array}{cc}
\alpha & \beta  \\
\gamma & \delta
\end{array} \right).
\end{eqnarray*}

\subsection{Basic properties of the matrix $A$}
In this section we collect some elementary facts about the matrix $A.$ \\
First note that we have explicit expressions for its coefficients:
\begin{align*}
\alpha &=  - \frac{C_V}{\kappa} \int_0 ^r P(r') \big(\frac{r'}{2} + U \big) dr',  \\
\beta &= -\frac{R}{\kappa} \int_0 ^r P \Theta dr' ,  \\
\gamma &= -R \int_0 ^r \frac{P}{2 \mu + \lambda} dr' , \\
\delta &= - \int_0 ^r \frac{P(r')}{2\mu+\lambda} \big(\frac{r'}{2} + U \big) dr'.
\end{align*}
\begin{rem}\label{notation}
Given the assumption $C_V \leqslant \frac{\kappa}{2\mu + \lambda},$ we have $\delta < \alpha.$ 
\end{rem}
The next lemma shows that the matrix $A$ can be diagonalized.
\begin{lem}\label{diag-A}
We have the following decomposition for $A:$ \\
There exists a real-valued $2 \times 2$ matrix $Q,$ with diagonal elements equal to one, such that:
\begin{eqnarray*}
A(r) = Q^{-1}(r)\left(
\begin{array}{cc}
- \lambda_{min} &  0 \\
0   &  -\lambda_{max}
\end{array}
 \right)  Q(r),
\end{eqnarray*}
where $-\lambda_{max}<-\lambda_{min}$ 
denote the two real valued eigenvalues of $A.$ 
\end{lem}
\begin{proof}
The characteristic polynomial of that matrix $A$ is $X^2 - (\delta + \alpha) X + \alpha \delta - \beta \gamma. $
\\
Its discriminant is 
\begin{align*}
\Delta &= (\delta + \alpha)^2 - 4 (\alpha \delta - \beta \gamma) = (\alpha - \delta)^2 + 4 \beta \gamma.
\end{align*}
Given that both $\beta$ and $\gamma$ are negative, $\Delta >0.$
\\
From this we can deduce the two eigenvalues:
\begin{align*}
-\lambda_{max}  :=  \frac{1}{2} \big( \alpha + \delta - \sqrt{\Delta} \big), \ \ \ \   -\lambda_{min}  :=  \frac{1}{2} \big(\alpha + \delta + \sqrt{\Delta}  \big) . 
\end{align*}
We find the corresponding eigenvectors, and deduce that the matrix $A(r)$ can then be diagonalized after introducing the matrix $Q(r)$ defined as 
\begin{eqnarray*} 
Q(r) &=& \left( \begin{array}{cc}
1     &       \frac{-\lambda_{max}-\delta}{\gamma}  \\
\frac{\gamma}{-\lambda_{min}-\delta}   &     1
\end{array} \right), \\
Q^{-1}(r) &=& \frac{1}{D} \left( \begin{array}{cc}
1     &      -   \frac{-\lambda_{max}-\delta}{\gamma}   \\
- \frac{\gamma}{-\lambda_{min}-\delta}  &     1
\end{array} \right),
\end{eqnarray*}
where $D:=1-\frac{-\lambda_{max} - \delta}{-\lambda_{min} - \delta}. $
\end{proof}
We write 
\[
Q(r) = 
\left( \begin{array}{cc}
1     &       q_{12}  \\
q_{21}   &     1
\end{array} \right), 
\qquad 
Q^{-1} (r) = 
\frac{1}{D} 
\left( \begin{array}{cc}
1     &      - q_{12}  \\
- q_{21}   &     1
\end{array} \right), 
\qquad 
D = 1 - q_{12} q_{21}.
\]
The following lemma justifies the fact that the matrix $Q$ is a perturbation of the identity for large values of $r$:
\begin{lem} \label{asymp}
We have the following estimates, valid for any $r>0:$
\begin{eqnarray*}
&\vert q_{12} \vert & \leqslant \frac{2(2\mu_0+\lambda_0) R \Lambda \sup_{r>0}P^{\alpha} \Theta}{\kappa \big(1/2-b\big) r}, \\
&\vert q_{21} \vert & \leqslant \frac{2 R \Lambda}{(1/2-b) r}, \\
&  \vert D \vert & \geqslant 1.
\end{eqnarray*}
\end{lem}
\begin{proof}
We have
\begin{align*}
\vert q_{12} \vert  &= \bigg \vert \frac{-\lambda_{max} - \delta}{\gamma} \bigg \vert = \frac{2 \beta}{\alpha - \delta + \sqrt{\Delta}} \leqslant \frac{\beta}{\alpha-\delta}.                 
\end{align*}
Using remark \ref{C_V} and the condition \eqref{cavitation}, we obtain the desired bound. \\
Similarly, we obtain the bound on $q_{21}.$ \\
\\
For the last estimate on $D$, we notice that
\begin{align*}
1-\frac{\alpha - \delta - \sqrt{\Delta}}{\alpha - \delta + \sqrt{\Delta}} = 1  - \frac{-4\beta \gamma}{(\alpha - \delta + \sqrt{\Delta})^2} \geqslant 1,
\end{align*}
since both $\beta$ and $\gamma$ are negative.
\end{proof}
In what follows, we will systematically decompose $Q$ as the sum of the identity matrix, and the matrix of its off-diagonal terms that will be treated like an error term. The following basic computation will be used repeatedly in the sequel:
\begin{lem} \label{elem}
We have, for real numbers $a,b,c,d$:
 \[
\begin{split}
& 
\left\langle Q^{-1}\left( \begin{array}{cc}
e^{-\lambda_{min}} & 0 \\
0                  & e^{-\lambda_{max}}
\end{array}\right)
Q
\left( \begin{array}{cc}
a \\
b
\end{array}\right), 
\left( \begin{array}{cc}
c \\
d
\end{array}\right)
\right\rangle 
\\
= 
& 
 \Big\{ \Big( 1 + \frac{q_{12}q_{21}}{D} \Big) e^{-\lambda_{min}} - \frac{q_{12}q_{21}}{D} e^{-\lambda_{max}}\Big\} ac 
\\
& 
+ \Big\{ \Big( 1 + \frac{q_{12}q_{21}}{D} \Big) e^{-\lambda_{max}} - \frac{q_{12}q_{21}}{D} e^{\lambda_{min}} \Big\} bd 
\\
& + (e^{-\lambda_{min}} - e^{-\lambda_{max} } ) \frac{q_{12} }{D} bc
  + (-e^{-\lambda_{min}} + e^{-\lambda_{max} } ) \frac{q_{21}}{D} ad. 
\end{split}
\]
We also have 
\[
\frac{1}{2} \leq 1 + \frac{q_{12}q_{21}}{D} \leq 1 , \quad 0 < -q_{12}q_{21} \leq 1.
\]
\end{lem}
We end this section with rough bounds that are relevant to treat the case where $r$ is close to 0:
\begin{lem}\label{asymp-small}
We have the following estimate for any $r>0:$
\begin{align*}
\vert q_{12} \vert \leqslant \sqrt{\frac{2\mu_0+\lambda_0}{\kappa} \sup_{r>0}P^{\alpha} \Theta}.
\end{align*}
Moreover, we have
\begin{align*}
\sup_{0<r<\varepsilon} e^{\sqrt{\Delta}}  \leqslant 2, \ \ \sup_{0<r<\varepsilon} e^{\lambda_{min}}  \leqslant 2.
\end{align*}
\end{lem}
\begin{proof}
We write that
\begin{align*}
\vert q_{12} \vert = \bigg \vert \frac{-\lambda_{max}-\delta}{\gamma} \bigg \vert &= \frac{2\vert \beta \vert}{\alpha - \delta + \sqrt{(\alpha- \delta)^2 + 4 \beta \gamma}} \leqslant \sqrt{\frac{\beta}{\gamma}} \leqslant \sqrt{\frac{2\mu_0+\lambda_0}{\kappa} \sup_{r>0}P^{\alpha} \Theta}. 
\end{align*}
For the last two estimates, we use the following crude bounds:
\begin{align*}
\sqrt{\Delta} & \leqslant \sqrt{4\frac{R^2}{\kappa} \frac{\varepsilon^2}{2\mu_0 + \lambda_0} \sup_{r>0} P \Theta \sup_{r>0} P^{1-\alpha} + \frac{\varepsilon^4}{(2\mu_0+\lambda_0)^2} \big( \sup_{r>0} P^{1-\alpha} \big)^2}, \\
\lambda_{min} & \leqslant \frac{\varepsilon^2}{2\mu_0+\lambda_0} \sup_{r>0} P^{1-\alpha}. 
\end{align*}
The estimates then follow from the smallness assumptions.
\end{proof}

\subsection{Weighted energy estimate}
Now we take the inner product the equation \eqref{eq} with $(\Theta ;U)^T,$ and integrate by parts in the left-hand side.
\\
Then we split the expression between small and large values of $r.$ Overall we get
\begin{eqnarray*}
\textrm{LHS} &=& \int_0 ^{\varepsilon} \langle \exp \big( A(r) \big) \left( \begin{array}{ll}
\Theta' \\
U'
\end{array} \right), \left( \begin{array}{ll}
\Theta' \\
U'
\end{array} \right) \rangle dx  \\
&+& \int_{\varepsilon}^{\infty} \langle \exp \big( A(r) \big) \left( \begin{array}{ll}
\Theta' \\
U'
\end{array} \right), \left( \begin{array}{ll}
\Theta' \\
U'
\end{array} \right) \rangle dx \\
&:=& \textrm{LHS}_1 + \textrm{LHS}_2.
\end{eqnarray*}
We prove the following estimate on the left-hand side: 
\begin{lem} \label{LHS}
We have 
\begin{align*}
\textrm{LHS} \geqslant \int_{0}^{\varepsilon} e^{-\lambda_{max}} \big( \Theta'^2 + U'^2 \big)~dx + \frac{1}{4} \Bigg( \int_{\varepsilon}^{+\infty} e^{-\lambda_{min}} \Theta'^2 ~dx +\int_{\varepsilon}^{+\infty} e^{-\lambda_{max}} U'^2 ~dx \Bigg).
\end{align*}
\end{lem}
\begin{proof}
\underline{Bound on $\textrm{LHS}_1:$}\\
In this case the bound is straightforward:
\begin{align*}
\textrm{LHS}_1 \geqslant \int_{0}^{\varepsilon} e^{-\lambda_{max}} \big( U'^2 + \Theta'^2 \big)~dx.
\end{align*}
\underline{Bound on $\textrm{LHS}_2:$}\\
As announced above, in this range, we use the diagonalization of $A$ and we write the matrices $Q$ and $Q^{-1}$ as the sum of the identity and an error term. This yields, using Lemma \ref{elem}:
\begin{align*}
\textrm{LHS}_2 &=  \int_{\varepsilon}^{\infty} \Big\{ \Big( 1 + \frac{q_{12}q_{21}}{D} \Big) e^{-\lambda_{min}} - \frac{q_{12}q_{21}}{D} e^{-\lambda_{max}}\Big\} \Theta'^2 dx \\
& +  \int_{\varepsilon}^\infty \Big\{ \Big( 1 + \frac{q_{12}q_{21}}{D} \Big) e^{-\lambda_{max}} - \frac{q_{12}q_{21}}{D} e^{\lambda_{min}} \Big\} U'^2~ dx \\
& + \int_{\varepsilon}^{\infty} (e^{-\lambda_{min}} - e^{-\lambda_{max} } ) \frac{q_{12} -q_{21} }{D} U' \Theta'~dx 
\end{align*}
We notice that 
\[
\frac{1}{2} \leq 1 + \frac{q_{12}q_{21}}{D} \leq 1 , \quad -q_{12}q_{21} > 0,
\]
which makes the first two terms the main contribution. The last term is treated as an error. 
\\
Using the bounds from Lemma \ref{asymp} and the smallness assumptions, we write that
\begin{align*}
\Bigg \vert \int_{\varepsilon}^{\infty} 
\frac{q_{12}}{D} |e^{-\lambda_{min}} - e^{-\lambda_{max} } |
U' \Theta'~dx \Bigg \vert & \leqslant \sup_{r \geqslant \varepsilon} \bigg \vert \frac{U'}{r \Theta'} \bigg \vert  \frac{2(2\mu_0+\lambda_0) R \Lambda \sup_{r>0}P^{\alpha} \Theta}{\kappa (1/2-b)} \int_{\varepsilon} ^{+\infty} e^{-\lambda_{min}} \Theta'^2 ~dx \\
& \leqslant \frac{1}{100}  \int_{\varepsilon} ^{+\infty} e^{-\lambda_{min}} \Theta'^2 ~dx, \\
\Bigg \vert \int_{\varepsilon}^{\infty} \frac{q_{21}}{D} 
|e^{-\lambda_{min}} - e^{-\lambda_{max} } | 
\Theta' U'~dx \Bigg \vert & \leqslant \sup_{r \geqslant \varepsilon}  \bigg \vert \frac{U'}{r \Theta'} \bigg \vert  \frac{2R\Lambda}{1/2 - b} \int_{\varepsilon} ^{+\infty} e^{-\lambda_{min}} \Theta'^2 ~dx \\
& \leqslant  \frac{1}{100}  \int_{\varepsilon} ^{+\infty} e^{-\lambda_{min}} \Theta'^2 ~dx.
\end{align*}
The other terms are easier to bound, therefore we omit the details.
\end{proof}
\noindent Now we move on to the right-hand side. We show that
\begin{lem}\label{RHS}
We have
\begin{align*}
\textrm{RHS} & \leqslant \frac{1}{20} \Bigg( \int_{\mathbb{R}^d} e^{-\lambda_{min}} \Theta'^2~dx + \int_{\mathbb{R}^d} e^{-\lambda_{max}} U'^2~dx - \int_{\mathbb{R}^d} e^{-\lambda_{max}} \frac{P}{2\mu+\lambda} U^2~dx  \Bigg) \\
& - \frac{C_V}{4\kappa} \int_{\mathbb{R}^d}e^{-\lambda_{min}} P \Theta^2~dx.
\end{align*}
\end{lem}
\begin{proof}
We start by decomposing the right-hand side into three parts:
\begin{eqnarray*}
&& 
\bigg\langle RHS \text{ of } \eqref{eq}, 
\left( \begin{array}{ll}
\Theta \\
U
\end{array} \right) \bigg\rangle 
\\
&=& - \bigg \langle \exp \big(A(r)\big) \left( \begin{array}{ll}
0
\\
\frac{d-1}{r^2} U
\end{array} \right) , \left( \begin{array}{ll}
\Theta \\
U
\end{array} \right) \bigg \rangle \\
&-& \bigg \langle \exp \big(A(r)\big)\left( \begin{array}{ll}
\displaystyle \frac{C_V}{\kappa} P(r) \Theta (r) 
\\
\displaystyle \frac{1}{2(2\mu+\lambda)} P(r) U(r)
\end{array}
\right),  \left( \begin{array}{ll}
\Theta \\
U
\end{array} \right) \bigg \rangle \\
&+& \bigg \langle \exp \big(A(r)\big) \left( \begin{array}{ll}
N_1(P,\Theta,U)
\\
N_2(P,\Theta,U)
\end{array}
\right), \left( \begin{array}{ll}
\Theta \\
U
\end{array} \right) \bigg \rangle \\
&:=& \textrm{RHS}_1 + \textrm{RHS}_2 + \textrm{RHS}_3,
\end{eqnarray*}
where
\begin{align*}
N_1 (P,\Theta,U) &= - \frac{R}{\kappa} P \Theta \frac{d-1}{r} U + 2 \frac{\mu}{\kappa} \bigg( (U')^2 + \big( \frac{d-1}{r} U \big)^2 \bigg) + \frac{\lambda}{\kappa} \bigg(U' + \frac{d-1}{r} U \bigg)^2  \\
&- \frac{2 \mu + \lambda}{\kappa}\alpha \frac{U' +  \frac{d-1}{r} U}{\frac{1}{2}r + U} U'U - \frac{\lambda}{\kappa} \alpha\frac{U' +  \frac{d-1}{r} U}{\frac{1}{2}r + U}\frac{d-1}{r} U^2, \\
N_2 (P,\Theta,U) &=  - \alpha \frac{U' +  \frac{d-1}{r} U}{\frac{1}{2}r + U} U' - \frac{\lambda}{2\mu+\lambda} \alpha \frac{U' + \frac{d-1}{r} U}{\frac{1}{2}r + U}  \frac{d-1}{r} U.
\end{align*}
Next we estimate $\textrm{RHS}_1, \textrm{RHS}_2$ and $\textrm{RHS}_3$. The result from Lemma \ref{RHS} will then follow. \\
\\
\underline{Bound on $\textrm{RHS}_1$:}\\
As we did above, we use the elementary computation from Lemma \ref{elem} to obtain the expression: 
\begin{align}
\textrm{RHS}_1 & \label{R11} = - \int_{\mathbb{R}^d} \Big\{ \Big( 1 + \frac{q_{12}q_{21}}{D} \Big) e^{-\lambda_{max}} - \frac{q_{12}q_{21}}{D} e^{\lambda_{min}} \Big\} \frac{d-1}{r^2} U^2~dx   \\
& \label{R12} - \int_{\mathbb{R}^d} (e^{-\lambda_{min}} - e^{-\lambda_{max} } ) \frac{q_{12} }{D}  \frac{d-1}{r^2} U \Theta~dx    
\end{align}
We know that the first term is negative. For the second term, 
we will distinguish between $r$ large and $r$ small. We introduce a cut-off function $\chi$ such that:
\begin{itemize}
\item $\chi = 1$ on $[0,\varepsilon].$
\item $\chi$ is supported on $[0,2\varepsilon].$
\item $\vert \chi ' \vert \leqslant C \textbf{1}_{\lbrace \varepsilon \leqslant \vert x \vert \leqslant 2 \varepsilon \rbrace}$ for some numerical constant $C.$ Here $\textbf{1}_A$ denotes the characteristic function of the set $A$
\end{itemize} 
Using this cut-off, we write:
\begin{align}
\vert \eqref{R12} \vert \leq 
            & \label{R12s} 
  \int_{\mathbb{R}^d} \frac{| q_{12}|}{D} e^{-\lambda_{min}} \frac{d-1}{r^2} \chi^2 |U| \Theta ~dx\\
            & + \label{R12l} \int_{\mathbb{R}^d} \frac{| q_{12}|}{D} e^{-\lambda_{min}} \frac{d-1}{r^2} (1-\chi^2) |U| \Theta ~dx.
\end{align}
\textit{Bound on \eqref{R12s}:}\\
We write, using Lemma \ref{asymp-small}, Young's inequality, as well as the mean value theorem, we obtain the bound: 
\begin{align*}
\vert \eqref{R12s} \vert & \leqslant \frac{1}{2} \sqrt{\frac{2\mu_0+\lambda_0}{\kappa} \sup_{r>0} P^{\alpha} \Theta} \sup_{0<r<\varepsilon} ( e^{2 \sqrt{\Delta}} ) \int_0 ^{+\infty} e^{-\lambda_{max}} \chi^2 \bigg[ \frac{(d-1) U^2}{r^2} + \frac{(d-1) \Theta^2}{r^2} \bigg]~dx.
\end{align*}
The first part containing $U$ can be absorbed into \eqref{R11} with the smallness assumptions and Lemma \ref{asymp-small}. \\ 
For the part containing $\Theta,$ we use that $-\lambda_{max} \leqslant 0$ as well as Hardy's inequality, and the properties of the cut-off $\chi$ to write:
\begin{align*}
\int_{\mathbb{R}^d} e^{-\lambda_{max}} \chi^2 \frac{(d-1)\Theta^2}{r^2}~ dx & \leqslant C' \int_{\mathbb{R}^d} \big((\chi \Theta)' \big)^2 ~dx \\
 & \leqslant 2 C' \int_{\mathbb{R}^d} \Theta'^2~ dx + 2 C' \int_{\mathbb{R}^d} \chi'^2 \Theta^2~dx \\
 & \leqslant 2 C' \sup_{0<r<\varepsilon} e^{\lambda_{max}} \int_{\mathbb{R}^d} e^{-\lambda_{max}} \Theta'^2~ dx \\
 &+ 2C' C^2 \sup_{0<r<\varepsilon} e^{\lambda_{min}} \frac{1}{P_{\varepsilon}} \int_{\mathbb{R}^d} e^{-\lambda_{min}} P \Theta^2 ~ dx ,
\end{align*}
where $C'$ denotes the constant in Hardy's inequality. \\
Given the smallness assumptions, these contributions are acceptable.
\\
\\
\textit{Bound on \eqref{R12l}:}\\
Using Lemma \ref{asymp-small}, we obtain:
\begin{align*}
\vert \eqref{R12l} \vert & \leqslant \sup_{r > \varepsilon} \bigg \vert \frac{U}{r\Theta} \bigg \vert \frac{2 R\Lambda(2\mu_0+\lambda_0) \sup_{r>0} P^{\alpha} \Theta}{\kappa (1/2-b)} \int_{\mathbb{R}^d} (1-\chi^2) e^{-\lambda_{min}} \frac{d-1}{r^2} \Theta^2 ~dx \\
                         & \leqslant \frac{2R\Lambda(2\mu_0+\lambda_0) \sup_{r>0} P^{\alpha} \Theta}{\varepsilon ^2 \kappa P_{\varepsilon} (1/2-b)} \sup_{r > \varepsilon} \bigg \vert \frac{U}{r\Theta} \bigg \vert \int_{\mathbb{R}^d} e^{-\lambda_{min}} P \Theta^2~dx ,
\end{align*}
which will be controlled by a part of ${\rm RHS}_2$. We can conclude with the smallness assumptions. 
\\
\\
\noindent \underline{Bound on $\textrm{RHS}_2$:} 
\\
We use a similar reasoning for this part. With Lemma \ref{elem}, we obtain:
\begin{align}
\textrm{RHS}_2 &=\label{R21}-\frac{C_V}{\kappa} \int_{\mathbb{R}^d} \Big\{ \Big( 1 + \frac{q_{12}q_{21}}{D} \Big) e^{-\lambda_{min}} - \frac{q_{12}q_{21}}{D} e^{-\lambda_{max}}\Big\} P \Theta^2 ~dx \\
& 
- \int_{\mathbb{R}^d} \frac{P}{2(2\mu+\lambda)} \Big\{ \Big( 1 + \frac{q_{12}q_{21}}{D} \Big) e^{-\lambda_{max}} - \frac{q_{12}q_{21}}{D} e^{\lambda_{min}} \Big\} U^2 ~dx \\
\label{R22}&+\int_{\mathbb{R}^d} 
\Big\{ (e^{-\lambda_{min}} - e^{-\lambda_{max} } ) \frac{1 }{D} 
     \Big( \frac{ q_{12}}{2(2\mu+\lambda)} - \frac{q_{21}C_V}{\kappa} \Big) 
\Big\} PU\Theta ~dx . 
\end{align}
Now we estimate \eqref{R22}.  
We only estimate the terms with the slowest exponential, the other two terms can clearly be estimated in the same way. \\
As we did previously, we distinguish between $r$ small and $r$ large. \\
We start with the slowest exponential term in \eqref{R22}:
\begin{align}
\int_{\mathbb{R}^d}  e^{-\lambda_{min}} \frac{1}{D} \frac{q_{12}}{2\mu+\lambda} P U \Theta~dx & = \label{R23s} \int_0 ^{\varepsilon} e^{-\lambda_{min}} \frac{q_{12}}{D} \frac{P}{2\mu+\lambda} U \Theta ~dx \\
& \label{R23l} + \int_{\varepsilon}^{+\infty} e^{-\lambda_{min}} \frac{q_{12}}{D} \frac{P}{2\mu+\lambda} U \Theta ~dx,
\end{align}
and we estimate both pieces separately.\\
 In the case where $r$ is small, we use Lemma \ref{asymp-small} (more precisely the bound $\vert q_{12} \vert \leqslant \sqrt{\frac{\beta}{\gamma}}$ in the proof), and get:
\begin{align*}
\vert \eqref{R23s} \vert & \leqslant \int_0 ^{\varepsilon} e^{-\lambda_{min}} \sqrt{\frac{\int_0 ^r P \Theta}{\kappa P^{\alpha} \int_0 ^r P^{1-\alpha}}} \frac{\sqrt{P}}{\sqrt{2\mu+\lambda}} \vert U \vert \sqrt{P} \Theta~dx  \\
                         & \leqslant \frac{1}{2} \int_0 ^{\varepsilon} e^{-\lambda_{min}} \sqrt{\frac{\int_0 ^r P \Theta}{\kappa P^{\alpha} \int_0 ^r P^{1-\alpha}}} \big( \frac{P}{2\mu+\lambda} U^2 + P \Theta^2 \big)~dx \\
                         & \leqslant \frac{1}{2} \int_0 ^{\varepsilon} e^{-\lambda_{min}} \sqrt{ \frac{\Lambda \sup_{r>0} \Theta}{\kappa}}  P \Theta^2~dx \\ 
                         &+ \frac{1}{2} \sup_{0<r<\varepsilon} e^{2\sqrt{\Delta}} \int_0 ^{\varepsilon} e^{-\lambda_{max}} \sqrt{ \frac{\Lambda \sup_{r>0} \Theta}{\kappa}}  \frac{P}{2\mu+\lambda} U^2 ~dx,
\end{align*}
and we can conclude with the smallness conditions.
\\
Now we move on to the second part, using Lemma \ref{asymp}:
\begin{align*}
\vert \eqref{R23l} \vert & \leqslant \frac{2R \Lambda \sup_{r>0} \Theta}{\kappa (1/2-b)} \sup_{r>\varepsilon} \bigg \vert \frac{U}{r \Theta} \bigg \vert \int_{\mathbb{R}^d} e^{-\lambda_{min}} P \Theta^2 ~dx,
\end{align*}
and we can conclude using the smallness assumptions. \\
For the remaining term, we write
\begin{align*}
\frac{C_V}{\kappa}\int_{\mathbb{R}^d} \frac{q_{21}}{D} e^{-\lambda_{min}} P \Theta U ~dx \leqslant \frac{2 R \Lambda C_V}{\kappa(1/2-b)} \sup_{r>0} \bigg \vert \frac{U}{r \Theta} \bigg \vert \int_{\mathbb{R}^d} e^{-\lambda_{min}} P \Theta^2 ~dx.
\end{align*}
\\
\underline{Bound on $\textrm{RHS}_3$:} \\
All these terms are treated as error terms, and the proofs are easier or similar to the above. Therefore we only show how to estimate the main terms. \\
\\
\textit{Main contribution from $N_1:$}\\
For the first term, we simply use Remark \ref{beta} and write
\begin{align*}
\int_{\mathbb{R}^d} \frac{R}{\kappa} e^{-\lambda_{min}} P \Theta^2 \frac{d-1}{r} U dr \leqslant \frac{R(d-1) b}{\kappa} \int_{\mathbb{R}^d} e^{-\lambda_{min}} P \Theta^2 dr.
\end{align*}
Next, we write
\begin{align*}
\int_{\mathbb{R}^d} \frac{\mu}{\kappa} e^{-\lambda_{min}} U'^2 \Theta~dx &\leqslant \frac{\mu_0}{\kappa} \sup_{r>0} P^{\alpha} \Theta \sup_{0<r<\varepsilon} e^{2\sqrt{\Delta}} \int_0 ^{\varepsilon} e^{-\lambda_{max}} U'^2 ~dx \\
&+ \frac{\mu_0}{\kappa} \sup_{r>\varepsilon} \bigg \vert \frac{U'}{r \Theta'} \bigg \vert \sup_{r>0} P^{\alpha} r^2 \Theta  \int_{\varepsilon}^{+\infty} e^{-\lambda_{min}} \Theta'^2 ~dx .
\end{align*}
This is acceptable given our smallness assumptions. \\
Similarly
\begin{align*}
\int_{\mathbb{R}^d} \frac{\mu}{\kappa} e^{-\lambda_{min}} \frac{(d-1)U^2}{r^2} \Theta ~dx & \leqslant \frac{\mu_0}{\kappa} \sup_{r>0} P^{\alpha} \Theta \sup_{0<r<\varepsilon} e^{2\sqrt{\Delta}}  \int_0 ^{\varepsilon} e^{-\lambda_{max}} \frac{(d-1)^2 U^2}{r^2} ~dx \\
& + \frac{\mu_0}{\kappa} \sup_{r>\varepsilon} P^{\alpha -1} \Theta \frac{U^2}{r^2 \Theta^2} \int_{\varepsilon} ^{+\infty} e^{-\lambda_{min}} P \Theta^2~dx.
\end{align*}
The other main terms are bounded in the same way.
\\
\\
\noindent 
\textit{Main contribution from $N_2:$} \\
We write that 
\begin{align*}
\int_{\mathbb{R}^d} e^{-\lambda_{max}} \alpha \frac{U}{r/2+U} U'^2 ~dx \leqslant \frac{\alpha b}{\frac{1}{2}-b} \int_{\mathbb{R}^d} e^{-\lambda_{max}} U'^2~dx.
\end{align*}
The other main terms are bounded in the same way.
\end{proof}
\begin{proof}[Conclusion of the proof of Theorem \ref{main}]
The result follows directly by putting together Lemmas \ref{LHS} and \ref{RHS}.
\end{proof}

\end{document}